\documentclass[preprint,10pt,]{elsarticle}
\usepackage{amsthm}
\usepackage{graphicx}
\usepackage{subfig}
\usepackage{wrapfig}
\usepackage{sidecap}
\usepackage{amssymb,amsmath}
\usepackage{stmaryrd}
\usepackage{pstricks}
\usepackage{pst-node}
\usepackage{pst-blur}
\usepackage{pst-plot}
\usepackage{multicol}
\usepackage{placeins}
\usepackage{stdg}
\DeclareFontFamily{OT1}{pzc}{}
\DeclareFontShape{OT1}{pzc}{m}{it}{<-> s * [1.10] pzcmi7t}{}
\DeclareMathAlphabet{\mathpzc}{OT1}{pzc}{m}{it}
\begin{document}
\begin{frontmatter}
\title{Non-dissipative space-time $hp$-discontinuous Galerkin method for the
time-dependent Maxwell equations }
\author[gsc,temf]{M.~Lilienthal\corref{cor1}}
\ead{lilienthal@gsc.tu-darmstadt.de}
\author[ifh]{S.M.~Schnepp}
\ead{sascha.schnepp@ifh.ee.ethz.ch}
\author[temf]{T.~Weiland}
\ead{thomas.weiland@temf.tu-darmstadt.de}
\cortext[cor1]{Corresponding Author}

\address[gsc]{Graduate School of Computational Engineering, Technische Universitaet Darmstadt, Dolivostrasse 15, 64293 Darmstadt, Germany}
\address[ifh]{ETH Zurich, Institute of Electromagnetic Fields (IFH), Gloriastrasse 35, 8092 Zurich, Switzerland}
\address[temf]{Institut fuer Theorie Elektromagnetischer Felder, Technische Universitaet Darmstadt, Schlossgartenstrasse 8, 64289 Darmstadt, Germany}
\begin{abstract}
A finite element method for the solution of the time-dependent Maxwell equations
in mixed form is presented. The method allows for local $hp$-refinement in space
and in time.  To this end, a space-time Galerkin approach is employed. In
contrast to the space-time DG method introduced in \cite{vegt_space_2002} test and trial
space do not coincide. This allows for obtaining a non-dissipative method.
To obtain an efficient implementation, a hierarchical tensor product basis in space and time is proposed.
This allows to evaluate the local residual with a complexity of
$\mathcal{O}(p^4)$ and $\mathcal{O}(p^5)$ for affine and non-affine elements,
respectively.
\end{abstract}

\end{frontmatter}

\newcommand\restr[2]{{
\left.\kern-\nulldelimiterspace 
#1 
\vphantom{\big|} 
\right|_{#2} 
}}
\newtheorem{theorem}{Theorem}[section]
\newtheorem{lemma}{Lemma}[section]
\section{Introduction}
The accurate solution of large scale electromagnetic problems, where short wavelengths
need to be resolved in large computational domains, remains a challenge. Examples include 
antenna design, broadband scattering problems or electrically large structures. 
Especially for problems, where dispersion errors dominate, high order methods have 
advantages. Furthermore, if $hp$-refinement is applied in a judicious way, 
it is possible to obtain exponential convergence, even for solutions, which are 
locally non-smooth \cite{schwab_hp_1999}. This can lead to drastical savings in terms 
of degrees of freedom. In the past decade, there has been a lot of research on 
discontinuous Galerkin (DG)  methods for Maxwell's equations, see e.g. 
\cite{hesthaven_nodal_2002, fezoui_convergence_2005, cohen_spatial_2006}. The
use of discontinuous finite element spaces, allows in a natural
way for $hp$-refinement without the use of special transition elements as it is the 
case for continuous finite element methods \cite{demkowicz_computing_2006,demkowicz_computing_2007}. 
Often, DG methods are chosen for the spatial part of the discretization,
whereas time is discretized with explicit time integrators. Due to the conditional stability 
of the resulting schemes, the time step size in this case is determined by the 
smallest element and also the degree of the approximating polynomials. Thus, 
the temporal resolution is dictated by stability, not by actual accuracy requirements. 
To overcome this problem, local time-stepping schemes \cite{cohen_spatial_2006,
 piperno_symplectic_2006, taube_high_2009} and locally implicit schemes 
 \cite{dolean_locally_2010,descombes_locally_2012} have been proposed. 
 For both approaches, good speedups on locally refined meshes are reported. However,
  refinement in time is necessary for stability, rather than accuracy requirements.\\
Another approach are space-time DG methods \cite{vegt_space_2002}. Since space and time 
are discretized simultaneously, $hp$-refinement in space and time can be introduced 
naturally. The space-time DG methods are unconditionally stable. Together with their 
high flexibility, these methods are well-suited for space-time adaptivity.
However, many of the previously introduced space-time DG methods are
dissipative.
Dissipation may become an issue, especially for low approximation orders. In this paper we propose a space-time
 finite element method which is discontinuous with respect to the spatial directions and
  continuous in time. For a similar approach, regarding the temporal part of the
  discretization, see the recent contribution
  \cite{griesmaier_discretization_2013}, where a $h$-version hybrid DG method in space, combined with 
  a global continuous Galerkin approach in time is presented for
  the accoustic wave equation. However, we obtain a method, which allows for local $hp$-refinement 
  in space \emph{and} time, is energy-conserving and unconditionally stable.
In section 2 the space-time finite element method for the time-dependent Maxwell equations
\begin{align}
\label{maxwell}
&\eps \b{E}_t  - \curl \b{H} = \b{J}\;\; \mbox{in}\;\Omega\times (0,T]
\nonumber\\
&\mu \b{H}_t + \curl \b{E}=0 \;\; \mbox{in} \;\Omega \times(0,T] \nonumber \\
&\b{n} \times \b{E} = \b{n}\times\b{g}\;\;\mbox{on} \;\partial\Omega
\times(0,T] \nonumber \\
&\b{E}=\b{E}_0,\;\;\b{H}=\b{H}_0 \;\mbox{in} \;\Omega,\;t=0
\end{align}
is described, where $\eps(\b{x}), \mu(\b{x}): \Omega \rightarrow \mathbb{R}$
denote the electric permittivity and magnetic permeability, respectively.
In sections 3 and 4 we discuss the stability and energy conservation property of
the method. In section 5 we present a matrix-free implementation of the
space-time residual. It can be efficiently evaluated within an interative solution 
procedure, such that the method computationally behaves similarly to an explicit method.
 Section 6 is devoted to numerical experiments, including fully space-time 
 $hp$-adaptive simulations.

\section{Description of the method}
\subsection{Function spaces}

We will denote vector valued functions spaces with bold
letters, e.g. $\bL{2}(D):=[L_2(D)]^3$.
We introduce the spaces
\begin{align*}
&\b{H}(curl,\Omega):=\{\b{v} \in \bL{2}(\Omega):\curl\b{v} \in
\bL{2}(\Omega)\},\\
&\b{H}_0(curl,\Omega):=\{\b{v} \in \bL{2}(\Omega):\curl\b{v} \in \bL{2}(\Omega),
\b{n}\times \b{v} = 0\;\mbox{on}\,\p \Omega \}.
\end{align*}
 For $\b{J}=0, \b{g}=0$ the Maxwell system admits a unique solution
 $\b{U}=\{\b{E},\b{H}\}$ in (see e.g. \cite{remaki_methodes_1999}):
\begin{align*} 
\b{V}:=C^0([0,T],\b{H}_0(curl,\Omega))\cap C^1([0,T],\bL{2}(\Omega))
\times C^0([0,T],\b{H}(curl,\Omega)) \cap C^1([0,T],\bL{2}(\Omega)).
\end{align*}
\subsection{Partitioning of the space-time domain}
For the derivation of the method, we only consider spatial meshes consisting of
hexahedra. Nevertheless, most of the presented work is applicable to
tetrahedral meshes, as well.\\

We divide the time axis in intervals $\I_n=(t_{n-1},t_n]$, and thus obtain a partitioning of the space-time cylinder $\I\times\Omega$ in time slabs  $\I_n\times\Omega$.
For each time slab, the spatial domain $\Omega$ is partitioned into
non-overlapping (hexahedral) elements $K$ resulting in a triangulation $\mathcal{T}_n(\Omega)$.
We require that $\mathcal{T}_{n}(\Omega)$ can be obtained by refinement of a 
coarse triangulation $\overline{\mathcal{T}}(\Omega)$.
The obtained macro-elements are further bisected in temporal direction
$\I_n\times K = \bigcup_{k=1}^{N_K} I_k^K \times K$, such that we obtain 
a partition of the time slab $\I_n\times\Omega$ in space-time elements $I_k^K \times K \in \mathcal{S}_n(\I_n\times\Omega)$. Here $\mathcal{S}_n(\I_n\times\Omega)$ denotes the resulting triangulation of the time slab.
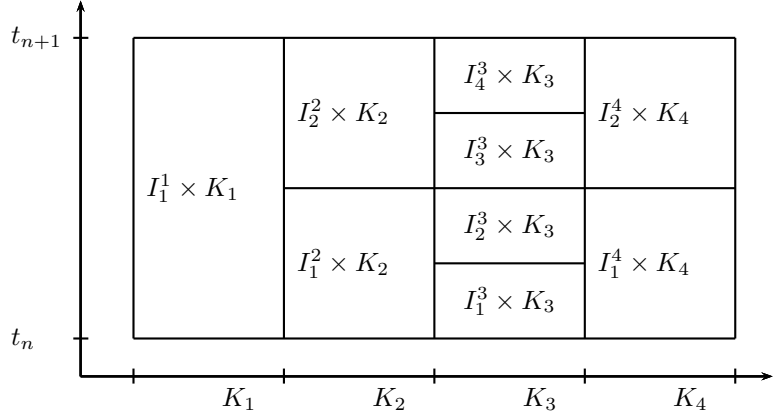
\begin{figure}[h!]

\begin{center}

\psset{unit=1cm}
\begin{pspicture}(-1.8,-1)(8,5)


\psaxes[labels=none,ticks=none]{->}(-0.7,-0.5)(-0.7,-0.5)(8.5,4.5)
\psline(8,0)(8,4)
\psline(0,0)(8,0)
\psline(0,4)(8,4)
\psline(0,0)(0,4)
\uput[0](-1.8,0){$t_{n}$}
\uput[0](-1.8,4){$t_{n+1}$}

\psline(-0.8,0)(-0.6,0)
\psline(-0.8,4)(-0.6,4)

\psline(6,0)(6,4)
\psline(4,0)(4,4)
\psline(2,0)(2,4)

\uput[0](1,-0.8){$K_1$}
\uput[0](3,-0.8){$K_2$}
\uput[0](5,-0.8){$K_3$}
\uput[0](7,-0.8){$K_4$}

\psline(6,2)(8,2)
\psline(4,3)(6,3)
\psline(4,2)(6,2)
\psline(4,1)(6,1)
\psline(2,2)(4,2)

\psline(0,-0.6)(0.,-0.4)
\psline(2,-0.6)(2.,-0.4)
\psline(4,-0.6)(4.,-0.4)
\psline(6,-0.6)(6.,-0.4)
\psline(8,-0.6)(8.,-0.4)
\uput[0](6,1){$I_1^4 \times K_4$}
\uput[0](6,3){$I_2^4 \times K_4$}

\rput{0}(5,0.5){$I_1^3 \times K_3$}
\rput{0}(5,1.5){$I_2^3 \times K_3$}
\rput{0}(5,2.5){$I_3^3 \times K_3$}
\rput{0}(5,3.5){$I_4^3 \times K_3$}

\uput[0](2,1){$I_1^2 \times K_2$}

\uput[0](2,3){$I_2^2 \times K_2$}

\uput[0](0,2){$I_1^1 \times K_1$}

\end{pspicture}
\vspace{4mm}
\caption{Space-time partitioning of a time-slab}

\end{center}
\end{figure}

\FloatBarrier
\subsection{Discrete spaces}
Let $F_K$ denote the mapping from the unit cube $\hat{K}=[0,1]^3$ with axes 
$\hat{x},\hat{y},\hat{z}$ to the physical element $K=F_K(\hat{K})$. Furthermore 
by $DF_K$ we denote the jacobian matrix of $F_K$. The electric and magnetic fields 
are transformed with the covariant transformation $\b{v}(x,t)=DF_K^{-T} \hat{\b{v}}(\hat{x},t)oF_K^{-1}$,
 as proposed in \cite{cohen_spatial_2006}.\\
By $P_{p_t}(I)$ we denote polynomials of degree $p_t$ on interval $I$ and by
$Q_{p_x,p_y,p_z}(\hat{K})$ tensor product polynomials of degrees $p_x, p_y$ and $p_z$ in the $\hat{x},\hat{y},\hat{z}$ directions.\\
Now we can introduce the following local discrete spaces
\begin{align}
\label{eq:spaces:local}
V^k_{h,\hat{K}}:= P_{p_t}(I_k^K)\otimes\left[Q_{p_x,p_y,p_z}(\hat{K})\right]^3\\
W^k_{h,\hat{K}}:= P_{p_t-1}(I_k^K)\otimes\left[Q_{p_x,p_y,p_z}(\hat{K})\right]^3.
\end{align}
$\b{p}_k^K=(p_t,p_x,p_y,p_z)$ denotes the local polynomial degree vector assigned to 
the space-time elements $I_k^K\times\hat{K}, k=1,\ldots N_K$. Note that the
spatial part of the polynomial spaces in \eqref{eq:spaces:local} is identical for 
all elements $I_k^K\times\hat{K}$ in one macro-element 
$\I_n\times\hat{K}$. Thus we obtain, for each macro-element,
tensor-product polynomial test- and trial-spaces consisting of a piecewise
polynomial temporal trial- and test spaces $S_K(\I_n)$ and
$T_K(\I_n)$ and the spatial part $Q_{p_x,p_y,p_z}(\hat{K})$.
\begin{align}
\label{eq:spaces:local:patch}
&V_{h,\hat{K}}:= S_K(\I_n) \otimes
\left[Q_{p_x,p_y,p_z}(\hat{K})\right]^3\hspace{-2mm},
\;\;\;S_K(\I_n):=\{u(t) \in H^1(\I_n) : \restr{u}{I_k^K} \in
P_{p_t}(I_k^K)\}\nonumber\\
&W_{h,\hat{K}}:= T_K(\I_n)
\otimes\left[Q_{p_x,p_y,p_z}(\hat{K})\right]^3\hspace{-2mm},
\;\;\;T_K(\I_n):=\{u(t) \in L_2(\I_n) : \restr{u}{I_k^K} \in P_{p_t-1}(I_k^K)\}.
\end{align}
We collect the local polynomial degree vectors for all elements in a global vector $\b{p}$.\\
Now we can define the global spaces for one time-slab
\begin{equation}
\label{eq:spaces:vglob}
V_h(\I_n\times\Omega;\b{p}):=\{\b{v}(t,\b{x}) \in H^1(\I_n;L_2(\Omega)) :  DF^T
\restr{\b{v}}{K} o F  \in V_{h,\hat{K}} \},
\end{equation}
and
\begin{equation}
\label{eq:spaces:vglob}
W_h(\I_n\times\Omega;\b{p}):=\{\b{v}(t,\b{x}) \in L_2(\I_n;L_2(\Omega)) : DF^T \restr{\b{v}}{K} o F  \in W_{h,\hat{K}} \}.
\end{equation}

While both spaces are  spatially discontinuous, the functions in $V_h$ and $W_h$ have different continuity properties in temporal direction: functions in $V_h$ are time-continuous within each time slab whereas functions in $W_h$ are allowed to be discontinuous at the interfaces in time direction. The situation is depicted for an example with three space-time elements in Fig. \ref{fig:continuity_test_trial}.
\begin{figure}[h!]
\begin{tabular}{cc}
\includegraphics[scale=0.35]{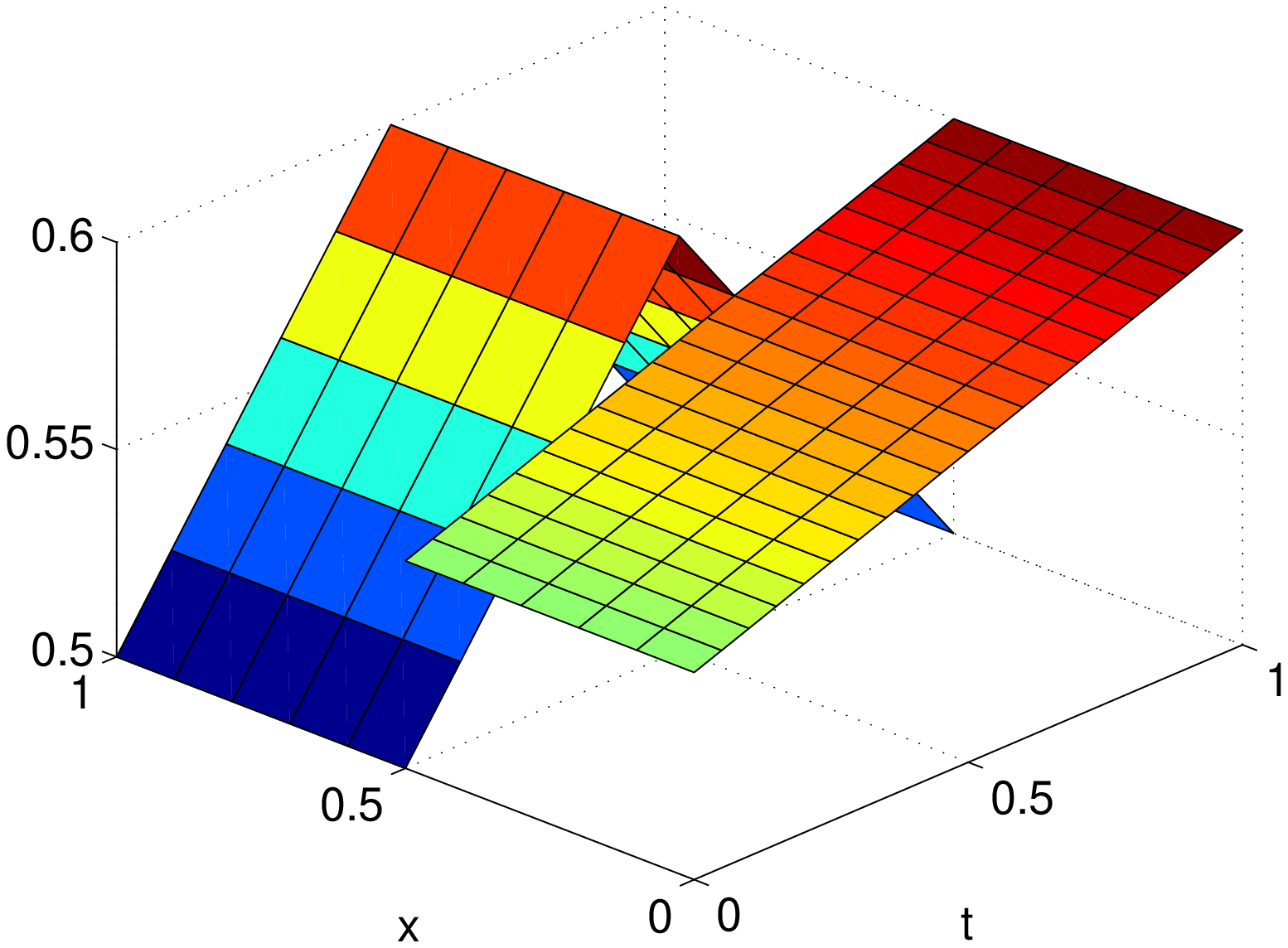} & \includegraphics[scale=0.35]{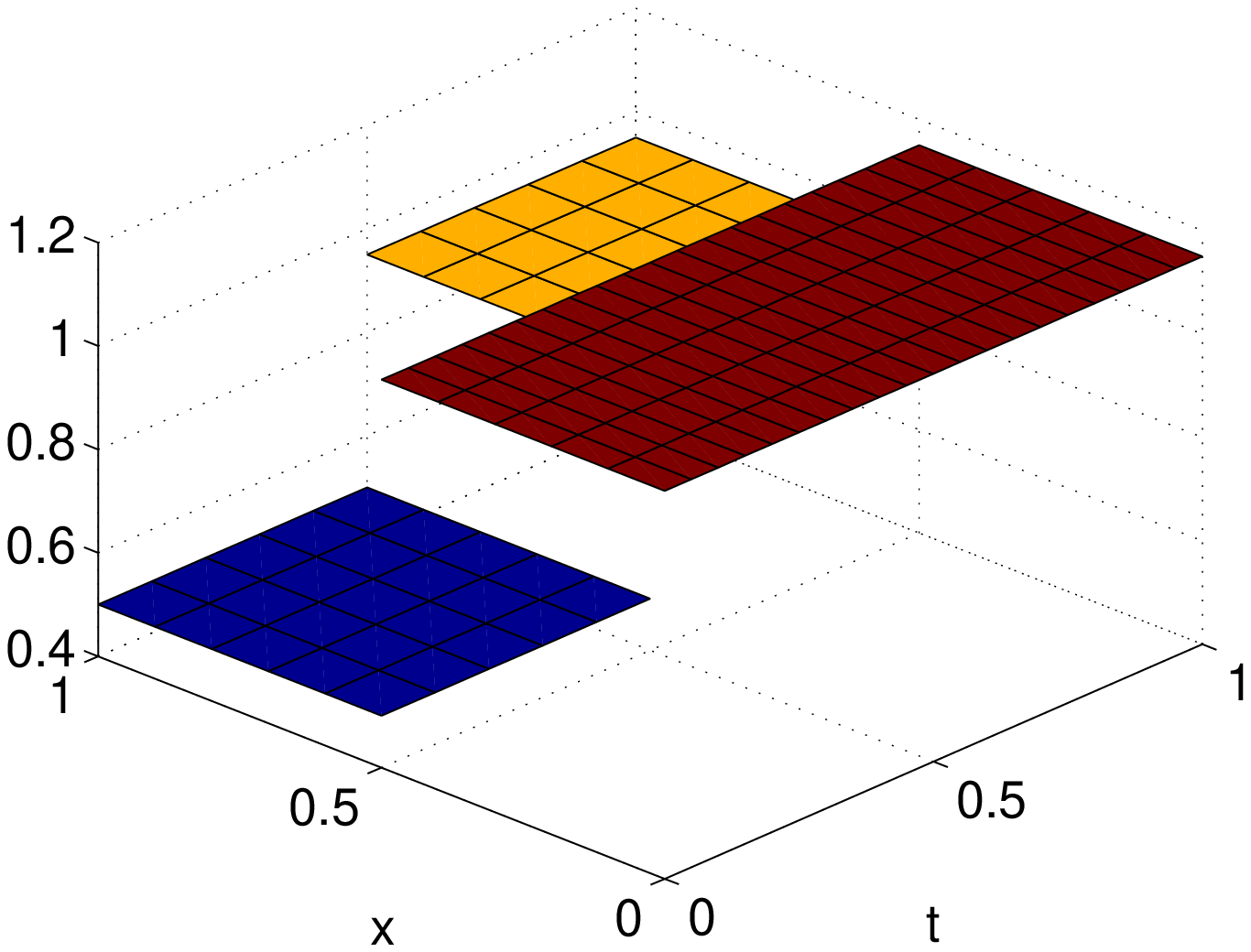}
\end{tabular}
\caption{Left: function belonging to trial space $V_h$, right: function from the corresponding test space $W_h$}\label{fig:continuity_test_trial}
\end{figure}
\FloatBarrier
\subsection{Faces and trace operators}
By $\mathcal{F}$ we denote the set of all faces in the spatial triangulation
$\mathcal{T}(\Omega)$, by $\Fi$ the set of all interior faces $f:=\p K^1 \cap \p
K^2 \;:\; K^1,\;K^2 \in \mathcal{T}(\Omega)$ and by $\mathcal{F}_{b}$ the set of all boundary faces $f:=\p K \cap \p\Omega\;:\;K\in\mathcal{T}(\Omega)$.\\
We define on interior faces $f\in\Fi$ the  average and tangential jump operators as $\{ \b{v}\}:=(\b{v}^1+\b{v}^2)/2$ and $\jumpt{\b{v}}:=\b{n}^1\times\b{v}^1 + \b{n}^2\times\b{v}^2$ respectively.
Here $\b{v}^1$ and $\b{v}^2$ denote the traces of $\b{v}$ on $f$ taken from within element $K^1$ and $K^2$ with unit normals $\b{n}^1$ and $\b{n}^2$.
For a boundary face $f\in\Fb$ we define averages and jumps as $\avg{\b{v}}:=\b{v}$ and $\jumpt{\b{v}}:=\b{n}\times\b{v}$.

In the following section we consider the discretization of a
single time-slab $\I_n \times \Omega$.
\subsection{Weak formulation}
Multiplying \eqref{maxwell} by smooth test functions $\b{v},
\b{w}$, integrating over a macro-element $\I_n \times K$ and performing
integration by parts of the terms involving the curl operator 
with respect to the spatial variables yields:
\begin{align}
& \inprodloc{\varepsilon \p_t\b{E}}{\b{v}}  - \inprodloc{\b{H}}{\curl \b{v}}\nonumber\\
& - \inprodf{\p K}{\b{n} \times \b{H}^*}{\b{v}} = \inprodloc{\b{J}}{\b{v}}\nonumber\\
& \inprodloc{\mu \p_t\b{H}}{\b{w}} + \inprodloc{\curl \b{E}}{\b{w}}\nonumber\\
& + \inprodf{\p K}{\b{n} \times (\b{E}^*-\b{E})}{\b{w}} =0.
\end{align}
Then by replacing  $\b{U}_h:=\{\b{E},\b{H}\} \in \b{V}$  by the discrete fields
$\b{E}_h,\b{H}_h \in V_h$, summing over all $\I_n \times K$ and choosing 
centered fluxes
\begin{align}
\b{n}\times\b{E}^*=\b{n}\times\avg{\b{E}_h}\;  f \in \mathcal{F}_0,
\hspace{0.2cm}
\b{n}\times\b{E}^*=\b{n}\times\b{g}\; f \in \mathcal{F}_D, \hspace{0.2cm}
\b{n}\times\b{H}^*=\b{n}\times\avg{\b{H}_h},\;  f \in \mathcal{F} \nonumber
\end{align}
one obtains 
\begin{align}\label{eq:bilinearform:C}
&C_h(\b{U_h},\b{V})=L(\b{V})\nonumber\\
&C_h(\b{U_h},\b{V}):=\inprod{\varepsilon \p_t\b{E}_h}{\b{v}} + \inprod{\mu \p_t\b{H}_h}{\b{w}}  + \nonumber\\
&- \inprod{\b{H}_h}{\curlh \b{v}} + \inprod{\curlh \b{E}_h}{\b{w}}\nonumber\\
&+ \inprodf{\Fi\cup\Fb}{\avg{ \b{H}_h}}{\jumpt{\b{v}}} -
\inprodf{\Fi\cup\Fb}{\jumpt{ \b{E}_h}}{\avg{\b{w}}}\\ &L(\b{V}):=
\ell_E(\b{v})\nonumber + \ell_H(\b{w})\nonumber\\
&\ell_E(\b{v}) = \inprod{\b{J}}{\b{v}},\;\;\; \ell_H(\b{w})=- \inprodf{\Fb}{\b{n}\times\b{g}}{\b{w}}.
\end{align}
\section{Stability - discretizations with global $hp$-refinement in time}
In this section we provide stability results in the $L_2(\Omega)$- and
$L_2(\I;\bL{2}(\Omega))$-norms.
In order to obtain the
stability bounds, we apply techniques similar to those presented in 
\cite{griesmaier_discretization_2013}.
 We restrict the analysis to
discretizations, where $T_K(\I_n)$ does not depend on the spatial element $K$,
i.e. no local refinement with respect to time is present.\\
We consider the problem
\begin{align}
\label{eq:weakform}
&\mbox{Find}\;\b{U}_h \in V_h\times V_h \quad\mbox{such that}\nonumber\\
&C_h(\b{U_h},\b{V}) =L(\b{V})\; \forall
\;\b{V}\;\in W_h \times W_h.
\end{align}
The general case, including local refinement in time, will be treated in the
next section.\\
\subsection{Stability: $L_2(\Omega)-Norm$}
Denoting the discrete electromagnetic energy by
\begin{align}
\label{stability:energydef}
\mathcal{E}(t) = \half \int_\Omega \left( \eps \b{E}_h \cdot\b{E}_h + \mu \b{H}_h \cdot \b{H}_h \right) \d{x},
\end{align}
we demonstrate stability of the method by showing, that the energy is constant up to a
contribution of the source terms. This result is a discrete version of Poyntings
theorem (see e.g. \cite{jackson_classical_1998}), which holds for the continuous
Maxwell system.
\newtheorem{stability:energy}{Theorem}[section]
\begin{stability:energy}\label{stability:energy:no_local_refinement}
Provided, the temporal polynomial degree $p_t$ is uniform for all $\I_n \times K$
and the material parameters $\eps, \mu$ are element-wise constant
\begin{align}\label{stability:energy:eq}
\mathcal{E}(t_{n+1}) - \mathcal{E}(t_{n})=\inprod{\b{J}}{\pvwE} - \inprodf{\Fb}{\b{n}\times\b{g}}{\pvwH}.
\end{align}
\end{stability:energy}
\begin{proof}
We denote by $\pvwE, \pvwH$ the $L_2$-orthogonal projections of the discrete
solution onto the test space $W_h$. Please note, that this projection reduces to
a projection with respect to time only, since the spatial parts of the
tensor-product trial- and testspaces coincide.\\
 Choosing $\b{v}=\pvwE,
\b{w}=\pvwH$ in \eqref{eq:weakform} yields
\begin{align}
\label{stability:weakform:tested:global_refinement}
&\inprod{\varepsilon \p_t\b{E}_h}{\pvwE} + \inprod{\mu \p_t\b{H}_h}{\pvwH} \nonumber\\
&- \inprod{\b{H}_h}{\curlh \pvwE} + \inprod{\curlh \b{E}_h}{\pvwH}\nonumber\\
& + \inprodf{\Fi\cup\Fb}{\avg{ \b{H}_h}}{\jumpt{\pvwE}} - \inprodf{\Fi\cup\Fb}{\jumpt{ \b{E}_h}}{\avg{\pvwH}} \nonumber\\
&= \inprod{\b{J}}{\pvwE} - \inprodf{\Fb}{\b{n}\times\b{g}}{\pvwH}
\end{align}
Since $\p_t \b{E}_h\in W_h$ and by using integration by parts, we obtain for the
first term in \eqref{stability:weakform:tested:global_refinement}
\begin{align*}
\inprod{\varepsilon \p_t\b{E}_h}{\pvwE} &= \inprod{\varepsilon \p_t\b{E}_h}{\b{E}_h}\\
&= \half \int_\Omega \eps \b{E}_h(t_{k+1}) \cdot \b{E}_h(t_{k+1}) \d{x} - \half \int_\Omega \eps \b{E}_h(t_{k}) \cdot \b{E}_h(t_{k}) \d{x}.
\end{align*}
Thus we obtain by treating the second
term in \eqref{stability:weakform:tested:global_refinement} the same way
\begin{align}
\label{eq:stability:energydifference}
\inprod{\varepsilon \p_t\b{E}_h}{\pvwE} + \inprod{\mu \p_t\b{H}_h}{\pvwH} = \mathcal{E}(t_{n+1}) - \mathcal{E}(t_{n}).
\end{align}
The third term yields
\begin{align*}
&- \inprod{\b{H}_h}{\curlh \pvwE} = - \inprod{\pvwH}{\curlh \b{E}_h}
\end{align*}
where we have used that $\projvw$ is an orthogonal projection and thus self-adjoint. Thus, the sum of the third and fourth terms is zero.\\
Now, consider the mesh-dependent terms associated with an interior face $f=\p
K^1\cap\p K^2\in\Fi$.
We have by a straightforward calculation
\begin{align}
\label{stability:fluxterms}
& + \inprodf{f}{\avg{ \b{H}_h}}{\jumpt{\pvwE}} -
\inprodf{f}{\jumpt{ \b{E}_h}}{\avg{\pvwH}} \nonumber\\
& =  \half \inprodf{f}{\b{H}_h^1}{\b{n}^1 \times \pi^1\b{E}^1} + \half
\inprodf{f}{\b{H}_h^2}{\b{n}^1 \times \pi^1\b{E}_h^1}\nonumber\\
& +  \half \inprodf{f}{\b{H}_h^1}{\b{n}^2 \times \pi^2\b{E}^2} +
\half \inprodf{f}{\b{H}_h^2}{\b{n}^2 \times \pi^2\b{E}^2}\nonumber\\
& -  \half \inprodf{f}{\pi^1\b{H}^1}{\b{n}^1 \times \b{E}_h^1} - \half
\inprodf{f}{\pi^1\b{H}^1}{\b{n}^2 \times \b{E}_h^2}\nonumber\\
& -  \half \inprodf{f}{\pi^2\b{H}^2}{\b{n}^1 \times \b{E}_h^1} -
\half \inprodf{f}{\pi^2\b{H}^2}{\b{n}^2 \times \b{E}_h^2}\\
&=T_1+T_2+T_3+T_4+T_5+T_6+T_7+T_8\nonumber,
\end{align}
where $\pi^1$ and $\pi^2$ denote the restrictions of the projection $\pi$ to local
spaces $W_{h,K^1}$ and $W_{h,K^2}$ from \eqref{eq:spaces:local:patch}.\\
First, we inspect the terms, which do not
couple to neighboring elements, for example $T_1$ and $T_5$
\eqref{stability:fluxterms}. Again, by the symmetry of $\pi^1$,
we have for the first term
\begin{align*}
T_1 = \half \inprodf{f}{\b{H}_h^1}{\b{n}^1 \times \pi^1\b{E}^1}
=   \half \inprodf{f}{\pi^1\b{H}^1}{\b{n}^1 \times \b{E}_h^1}=-T_5,
\end{align*}
such that $T_1+T_5=0$. Analogously we obtain $T_4+T_8=0$.\\
For the terms involving neighbor-coupling there holds, for example,
\begin{align*}
T_2=\half \inprodf{f}{\b{H}_h^2}{\b{n}^1 \times \pi^1\b{E}^1}= \half
\inprodf{f}{\pi^1 \b{H}_h^2}{\b{n}^1 \times \b{E}_h^1}\\
= \half \inprodf{f}{\pi^2 \b{H}_h^2}{\b{n}^1 \times \b{E}_h^1} =-T_7,
\end{align*}
thus, we obtain $T_2+T_7=0$ and similarly $T_3+T_6=0$.\\
Note, that this holds only due to $\pi^1\b{H}_h^2=\pi^2\b{H}_h^2$, which is fullfilled
since the temporal parts $T_{K^1}(\I_n)$ and $T_{K^2}(\I_n)$ of
the local test spaces $W_{h,K^1}$ and $W_{h,K^2}$ are
identical for all elements, which in turn is achieved by the restriction to
discretizations without local refinement in time.\\
Noting that terms associated with a boundary face can be treated exactly the same way
as the non-coupling terms, yields the desired result.
\end{proof}
\subsection{Stability $L_2(\I;L_2(\Omega))$-Norm}
For the special case of no local refinement in time we can also show the stability in the space-time $L_2$-norm $\|\cdot\|_{L_2(\I;\bL{2}(\Omega))}$.
First we recall the recurrence relations for the Legendre-polynomials $L_i(\xi)$
\begin{align}
\label{legendre_recurrence}
(i+1)L_{i+1}(\xi)=(2i+1)\xi L_i(\xi)-i\L_{i-1}(\xi)\nonumber\\
L_i^\prime(\xi)= 2L_{i-1}(\xi)/\|L_{i-1}\|_{L_2([-1,1])}^2 +
2L_{i-3}(\xi)/\|L_{i-3}\|_{L_2([-1,1])}^2 +\;...\nonumber\\
\xi L_{i+1}^\prime=(i+1)L_{i+1}(\xi)+iL_{i-1}(\xi) + (i-1)L_{i-1}(\xi)+(i-2)L_{i-3}(\xi)\;...
\end{align}

\begin{lemma}\label{stability:boundSTL2:lemma}
For element wise constant $\eps, \mu$ there holds
\begin{align}
\frac{1}{2\Delta t} \|\eps^\half
\b{E}_h\|_{L_2(\I_n;\bL{2}(\Omega))}^2 + \frac{1}{2\Delta t}
\|\mu^\half \b{H}_h\|_{L_2(\I_n;\bL{2}(\Omega))}^2  \leq \inprod{\eps \p_t \b{E}_h}{\projvw (\tau(t) \projvw \b{E}_h) }\nonumber\\
+ \inprod{\mu \p_t \b{H}_h}{\projvw (\tau(t) \projvw \b{H}_h) }+\mathcal{E}(t_k),\;\;\;      \tau(t)=\frac{(t_{k+1}-t)}{\Delta t}\nonumber
\end{align}
\end{lemma}
\begin{proof}
We have by integration by parts with respect to the temporal variable
\begin{align*}
\inprodloc{\eps \p_t \b{E}_h}{\projvw (\tau(t)\b{E}_h)  } =
\frac{1}{2\Delta t}
\|\eps^\half\b{E}_h\|_{L_2(\I_n;\bL2(\Omega))}^2 - \half \| \eps^\half
\b{E}_h(t_k)\|_K^2
\end{align*}
using that $\projvw$ is an orthogonal projection and $\p_t \b{E}_h\in W_h$ we can rewrite
\begin{align}
\label{stability:boundSTL2:lemma:split}
\inprodloc{\eps \p_t \b{E}_h}{\projvw (\tau(t)\b{E}_h)  } =
\inprodloc{\eps \p_t \b{E}_h}{\tau(t)( \b{E}_h-\projvw\b{E}_h)}\nonumber\\
+ \inprodloc{\eps \p_t \b{E}_h}{\tau(t)\projvw\b{E}_h}.
\end{align}
We will now show that the first term on the right-hand side of \eqref{stability:boundSTL2:lemma:split} is non-positive.
Due to the space-time tensor product construction of the local finite element space, we can expand the discrete solution and the projection error as
\begin{align}
\label{stability:boundSTL2:lemma:expansion}
\b{E}_h=\sum\limits_{i=0}^{p_t} \sum\limits_{k=1}^{N_s}L_i(\xi) \bs{\varphi}_k(x,y,z) e_{ik}\nonumber\\
\b{E}_h - \projvw\b{E}_h=\sum\limits_{k=1}^{N_s}L_{p_t}(\xi) \bs{\varphi}_k(x,y,z) e_{p_t\,k}
\end{align}
with  $\xi=2(t-t_k)/\Delta t - 1$ and $N_s=3(p_x+1)(p_y+1)(p_z+1)$.
Using once more the projection property, inserting the expansions \eqref{stability:boundSTL2:lemma:expansion}  we obtain
\begin{align*}
&\inprodlocst{\eps \p_t \b{E}_h}{\tau(t)( \b{E}_h-\projvw\b{E}_h)}=-\frac{\eps}{\Delta t} \inprodlocst{\p_t \b{E}_h}{t( \b{E}_h-\projvw\b{E}_h)}\\
&= -\frac{\eps}{\Delta t}  \sum\limits_{i=1}^{p_t}  \sum\limits_{j=1}^{N_s}\sum\limits_{l=1}^{N_s} \int\limits_{-1}^1 (t_k + \Delta t (1+\xi)/2)  L_i^\prime(\xi) L_{p_t}(\xi)  \d{\xi}\nonumber\\
&\times \int_K\bs{\varphi}_j(x,y,z) \cdot \bs{\varphi}_l(x,y,z) \d{x} e_{ij} e_{p_t\,l}\\
&= -\frac{\eps}{\Delta t}  \sum\limits_{i=1}^{p_t}  \int\limits_{-1}^1 \frac{\Delta t}{2} \xi  L_i^\prime(\xi) L_{p_t+1}(\xi)  \d{\xi} \sum\limits_{j=1}^{N_s}\sum\limits_{l=1}^{N_s}\int_K\bs{\varphi}_j(x,y,z) \cdot \bs{\varphi}_l(x,y,z) \d{x} e_{ij} e_{p_t\,l}.
\end{align*}
In the last step we have used the recurrence relation for the derivatives of  Legendre polynomials.
Finally, using the third recurrence relation in \eqref{legendre_recurrence}, we obtain by the orthogonality of the Legendre polynomials
\begin{align*}
\inprodlocst{\eps \p_t \b{E}_h}{\tau(t)( \b{E}_h-\projvw\b{E}_h)}\\
= - (\eps/2) \int\limits_{-1}^{1} L_{p_t}(\xi)^2\d{\xi} \sum\limits_{j=1}^{N_s}\sum\limits_{l=1}^{N_s} \int_K\bs{\varphi}_j(x,y,z) \cdot \bs{\varphi}_l(x,y,z) \d{x} \,e_{p_t+1\,j} e_{p_t+1\,l} \leq 0
\end{align*}
\end{proof}

Denoting the dual norm on the discrete test space by
\begin{align*}
\|\ell \|_{W_h^\prime}:=\sup\limits_{\b{v}\in W_h} \frac{|\ell(\b{v})|}{\|\b{v}\|_{L_2(\I_n;\bL{2}(\Omega))}},
\end{align*}
we can show

\begin{lemma}
Provided, the temporal polynomial degree $p_t$ is constant and no local $h$-refinement with respect to time is present in the discretization, for elementwise constant $\eps \geq \epsmin > 0$ and $\mu \geq \mumin > 0$
there holds
\begin{align}\label{stability:boundSTL2lemma2:eq}
\|\eps^\half \b{E}_h\|_{L_2(\I_n;\bL{2}(\Omega))}^2 + \|\mu^\half \b{H}_h\|_{L_2(\I_n;\bL{2}(\Omega))}^2 \leq   4 \Delta t^2 \left( \epsmin^{-1} \|\ell_E\|_{W_h^\prime}^2 + \mumin^{-1} \|\ell_H\|_{W_h^\prime}^2 \right) + 4 \Delta t \mathcal{E}(t_k)
\end{align}
\end{lemma}
\begin{proof}
Choosing  $\b{v}=\projvw (\tau(t) \pvwE), \b{w}= \projvw (\tau(t) \pvwH)$ in
\eqref{eq:weakform} yields
\begin{align}
\label{stability:boundSTL2:weakform:tested}
&\inprod{\varepsilon \p_t\b{E}_h}{\projvw (\tau(t) \pvwE)} + \inprod{\mu
\p_t\b{H}_h}{\projvw (\tau(t) \pvwH)} \nonumber\\
&- \inprod{\b{H}_h}{\curlh \projvw (\tau(t) \pvwE)} + \inprod{\curlh
\b{E}_h}{\projvw (\tau(t) \pvwH)}\nonumber\\
& + \inprodf{\Fi\cup\Fb}{\avg{ \b{H}_h}}{\jumpt{\projvw (\tau(t) \pvwE)})}
- \inprodf{\Fi\cup\Fb}{\jumpt{ \b{E}_h}}{\avg{\projvw (\tau(t) \pvwH)})}
\nonumber\\
&= \ell_E(\projvw (\tau(t) \pvwE)) + \ell_H(\projvw (\tau(t) \pvwH)).
\end{align}

Following the line of arguments of the proof of Theorem \ref{stability:energy:no_local_refinement}, 
all terms, except the first two terms on the left-hand side of 
\eqref{stability:boundSTL2:weakform:tested} vanish. \\
Applying lemma \ref{stability:boundSTL2:lemma}, the Cauchy-Schwarz inequality, 
the arithmetic-geometric-mean inequality yields
\begin{align*}
\frac{1}{2\Delta t}\left( \|\eps^\half
\b{E}_h\|_{L_2(\I_n;\bL{2}(\Omega))}^2 + \|\eps^\half
\b{H}_h\|_{L_2(\I_n;\bL{2}(\Omega))}^2 \right)\\
\leq  \ell_E(\projvw (\tau(t) \pvwE)) + \ell_H(\projvw (\tau(t) \pvwH)) + \mathcal{E}(t_k)\\
\leq \|\ell_E\|_{W_h^\prime} \| \projvw (\tau(t)
\pvwE)\|_{L_2(\I_n;\bL{2}(\Omega))}
+  \|\ell_H\|_{W_h^\prime} \|
\projvw (\tau(t) \pvwH)\|_{L_2(\I_n;\bL{2}(\Omega))} + \mathcal{E}(t_k)
\\
\leq \Delta t \left( \epsmin^{-1}\|\ell_E\|_{W_h^\prime}^2  + \mumin^{-1} \|\ell_H\|_{W_h^\prime}^2 \right)  + \mathcal{E}(t_k)\\
+ \frac{\epsmin}{4\Delta t}  \| \projvw (\tau(t)
\pvwE)\|_{L_2(\I_n;\bL{2}(\Omega))}^2 +  \frac{\mumin}{4\Delta t}\|
\projvw (\tau(t) \pvwH)\|_{L_2(\I_n;\bL{2}(\Omega))}^2\\
\leq \Delta t \left( \epsmin^{-1}\|\ell_E\|_{W_h^\prime}^2  + \mumin^{-1} \|\ell_H\|_{W_h^\prime}^2 \right)  + \mathcal{E}(t_k)\\
+ \frac{1}{4\Delta t} \left( \| \eps^\half \projvw (\tau(t)
\pvwE)\|_{L_2(\I_n;\bL{2}(\Omega))}^2 +  \| \mu^\half \projvw (\tau(t)
\pvwH)\|_{L_2(\I_n;\bL{2}(\Omega))}^2 \right).
\end{align*}
Since $\|\projvw\|\leq 1$ and $1\geq\tau(t)\geq0$ on $\I_n$, we obtain the result.
\end{proof}
\newtheorem{stability:boundSTL2theorem}[stability:energy]{Theorem}
\begin{stability:boundSTL2theorem}\label{stability:boundSTL2:theorem}
Provided the temporal polynomial degree $p_t$ is constant and no local $h$-refinement with respect to time is present in the discretization
there holds for elementwise constant $\eps \geq \epsmin > 0$ and $\mu \geq \mumin > 0$
\begin{align}
\mathcal{E}(t_n) \leq 2\mathcal{E}(t_1) + 2\sum\limits_{n=1}^{N} \left(
(\frac{2t_n}{\epsmin} + \frac{\Delta t^2}{2t_n\epsmin})\|\ell_E^k\|_{W_h^\prime}^2  + (\frac{2t_n}{\mumin} + \frac{\Delta t^2}{2t_n\mumin})\|\ell_H^k\|_{W_h^\prime}^2 \right)
\end{align}
\end{stability:boundSTL2theorem}
\begin{proof}
The proof follows along the lines of \cite{griesmaier_discretization_2013} Corollary 1.\\
Denoting $n_{max}=\argmax\limits_n \mathcal{E}(t_n)$, we have by
Theorem \ref{stability:energy:no_local_refinement} and the Cauchy-Schwarz
inequality, the arithmetic geometric mean inequality, \eqref{stability:boundSTL2lemma2:eq}
\begin{align*}
\max\limits_n \mathcal{E}(t_n) \leq \mathcal{E}(t_1) +
\sum\limits_{n=1}^{n_{max}} \left( \|\ell_E\|_{W_h^\prime} \|
\b{E}_h\|_{L_2(\I_n;\bL{2}(\Omega))}  + \|\ell_H\|_{W_h^\prime} \|
\b{H}_h\|_{L_2(\I_n;\bL{2}(\Omega))} \right)\\
\leq \mathcal{E}(t_1) + \sum\limits_{n=1}^{n_{max}} \left(
\frac{\delta}{\epsmin} \|\ell_E\|_{W_h^\prime}^2 + \frac{\delta}{\mumin}
\|\ell_H\|_{W_h^\prime}^2  + \frac{\epsmin}{4\delta}\|
\b{E}_h\|_{L_2(\I_n;\bL{2}(\Omega))}^2  +  \frac{\mumin}{4\delta}\|\b{H}_h\|_{L_2(\I_n;\bL{2}(\Omega))}^2 \right)\\
\leq \mathcal{E}(t_1) + \sum\limits_{n=1}^{n_{max}} \left(
(\frac{\delta}{\epsmin} + \frac{\Delta
t^2}{\epsmin\delta})\|\ell_E\|_{W_h^\prime}^2 + (\frac{\delta}{\mumin} + \frac{\Delta t^2}{\mumin\delta})\|\ell_H\|_{W_h^\prime}^2  +  \frac{\Delta t}{\delta} \mathcal{E}(t_n) \right).
\end{align*}
Chosing $\delta=2t_{n_{max}}$ we have
\begin{align*}
\sum\limits_{n=1}^{n_{max}} \frac{\Delta t}{2t_{n_{max}}} \mathcal{E}(t_n) \leq
\half \max\limits_n \mathcal{E}(t_n),
\end{align*}
and thus the estimate for $\mathcal{E}(t_n)$, which can be applied to the corresponding term in \eqref{stability:boundSTL2lemma2:eq}.
\end{proof}

\section{Stability - discretizations with local $hp$-refinement in time}
In this section we show, that stability in the $L_2(\Omega)$-norm can be obtained also
for discretizations with local $hp$-refinement in time.
This is achieved by adding a
suitable stabilization term, which amounts to restoring anti-symmetry in the
coupling flux terms.
\subsection{Stability: $L_2(\Omega)-Norm$}
For an interior space-time face
$\I_n\times f, f \in \mathcal{F}_0$ shared by $\I_n \times K^i, i=1,2$,
we denote by $_f : L_2(\I_n) \rightarrow \widetilde{T}(\I_n)$ the
$L_2$-orthogonal projection operator onto the largest common temporal testspace $\widetilde{T}(\I_n):=T_{K^1}(\I_n) \cap
T_{K^2}(\I)$. Further we denote by $\projt$ the projector whose restriction to 
$\I_n\times f$ is  $\projt_f$.\\
We add the stabilization form
\begin{align}
\label{eq:bilinearform:S}
&S_h(\b{U_h},\b{V}):=\half \left( \inprodf{\Fi}{\avg{ \b{H}_h}}{\jumpt{\projt \b{v}-\b{v}}} + \inprodf{\Fi}{\jumpt{ \b{H}_h}}{\avg{\projt \b{v}-\b{v}}} \right)\nonumber\\
& - \half \left( \inprodf{\Fi}{\avg{ \b{E}_h}}{\jumpt{\projt \b{w}-\b{w}}} + \inprodf{\Fi}{\jumpt{ \b{E}_h}}{\avg{\projt \b{w}-\b{w}}}\right).
\end{align}
and obtain the stabilized discrete problem
\begin{align}
\label{eq:weakform_stabilized}
&\mbox{Find}\;\b{U}_h \in V_h\times V_h\nonumber\\
&B_h(\b{U_h},\b{V}):=C_h(\b{U_h},\b{V}) + S_h(\b{U_h},\b{V})=L(\b{V})\; \forall
\;\b{V}\;\in W_h \times W_h.
\end{align}
Note, that in the case of globally in time refined discretization, as considered in the
previous section, we obtain $S_h(\b{U_h},\b{V})=0$ and   \eqref{eq:weakform_stabilized}
reduces to \eqref{eq:weakform}.\\
Now we can generalize theorem \ref{stability:energy:no_local_refinement} to
discretizations with local $hp$-refinement in time:
\begin{theorem}\label{stability:energy}
For element-wise constant material parameters $\eps, \mu$
\begin{align}\label{stability:energy:eq}
\mathcal{E}(t_{n+1}) - \mathcal{E}(t_{n})=\inprod{\b{J}}{\pvwE} - \inprodf{\Fb}{\b{n}\times\b{g}}{\pvwH}.
\end{align}
\end{theorem}
\begin{proof}
Regarding the treatment of the volume integrals and boundary faces, the proof is
identical to the one of theorem \ref{stability:energy:no_local_refinement}.
Thus, in the following we will only consider the mesh-dependent terms associated with interior faces $f\in\Fi$.\\
Again, we denote by $\pvwE, \pvwH$ the
$L_2$-projections of the discrete solution onto the test space. Choosing  $\b{v}=\pvwE,
\b{w}=\pvwH$ in \eqref{eq:weakform_stabilized} yields
\begin{align}
\label{stability:weakform:tested}
&\inprod{\varepsilon \p_t\b{E}_h}{\pvwE} + \inprod{\mu \p_t\b{H}_h}{\pvwH} \nonumber\\
&- \inprod{\b{H}_h}{\curlh \pvwE} + \inprod{\curlh \b{E}_h}{\pvwH}\nonumber\\
& + \inprodf{\Fi\cup\Fb}{\avg{ \b{H}_h}}{\jumpt{\pvwE}} - \inprodf{\Fi\cup\Fb}{\jumpt{ \b{E}_h}}{\avg{\pvwH}} \nonumber\\
&+ \half \left( \inprodf{\Fi}{\avg{ \b{H}_h}}{\jumpt{\projt \pvwE-\pvwE}} + \inprodf{\Fi}{\jumpt{ \b{H}_h}}{\avg{\projt \pvwE-\pvwE}} \right)\nonumber\\
& - \half \left( \inprodf{\Fi}{\avg{ \b{E}_h}}{\jumpt{\projt \pvwH-\pvwH}} + \inprodf{\Fi}{\jumpt{ \b{E}_h}}{\avg{\projt \pvwH-\pvwH}}\right)\nonumber\\
&= \inprod{\b{J}}{\pvwE} - \inprodf{\Fb}{\b{n}\times\b{g}}{\pvwH}
\end{align}
Algebraic manipulations lead to
\begin{align}
\label{stability:fluxterms_stabilized}
&  \inprodf{f}{\avg{ \b{H}_h}}{\jumpt{\pvwE}} -
\inprodf{f}{\jumpt{ \b{E}_h}}{\avg{\pvwH}} \nonumber\\
&+ \half \left( \inprodf{f}{\avg{ \b{H}_h}}{\jumpt{\projt_f \pvwE-\pvwE}} + \inprodf{f}{\jumpt{ \b{H}_h}}{\avg{\projt_f \pvwE-\pvwE}} \right)\nonumber\\
& - \half \left( \inprodf{f}{\avg{ \b{E}_h}}{\jumpt{\projt_f \pvwH-\pvwH}} + \inprodf{f}{\jumpt{ \b{E}_h}}{\avg{\projt_f \pvwH-\pvwH}}\right)\nonumber\\
& =  \half \inprodf{f}{\b{H}_h^1}{\b{n}^1 \times \pi^1\b{E}_h^1} + \half
\inprodf{f}{\b{H}_h^2}{\b{n}^1 \times \projt_f \pi^1\b{E}_h^1}\nonumber\\
& +  \half \inprodf{f}{\b{H}_h^1}{\b{n}^2 \times \projt_f \pi^2\b{E}_h^2} +
\half \inprodf{f}{\b{H}_h^2}{\b{n}^2 \times \pi^2\b{E}_h^2}\nonumber\\
& -  \half \inprodf{f}{\pi^1\b{H}_h^1}{\b{n}^1 \times \b{E}_h^1} - \half
\inprodf{f}{\projt_f \pi^1\b{H}_h^1}{\b{n}^2 \times \b{E}_h^2}\nonumber\\
& -  \half \inprodf{f}{\projt_f\pi^2\b{H}_h^2}{\b{n}^1 \times \b{E}_h^1} -
\half \inprodf{f}{\pi^2\b{H}_h^2}{\b{n}^2 \times \b{E}_h^2}\\
&=T_1+T_2+T_3+T_4+T_5+T_6+T_7+T_8\nonumber,
\end{align}
with $\pi^1$ and $\pi^2$ as in \eqref{stability:fluxterms}.
Comparing \eqref{stability:fluxterms_stabilized} with the corresponding terms in the
non-stabilized case \eqref{stability:fluxterms}, we observe, that the non-coupling
terms are identical, leading to $T_1+T_5=0$ and $T_4+T_8=0$. Only the coupling
terms, are modified due to the stabilization form.\\
We have
\begin{align*}
T_2= \half \inprodf{f}{\b{H}_h^2}{\b{n}^1 \times \projt_f \pi^1\b{E}_h^1}= \half
\inprodf{f}{\projt_f \b{H}_h^2}{\b{n}^1 \times \b{E}_h^1}
\end{align*}
and  $\pi^1 \projt_f\b{H}^2_h=\projt_f\b{H}^2_h$, since the temporal part of
$\projt_f\b{H}^2_h$ belongs to $T_{K^1}(\I_n)$.\\
Further, again using the symmetry of $\projt_f$ and $\pi^2$ we obtain
\begin{align*}
T_7=-\half \inprodf{f}{\projt_f \pvwH^2}{\b{n}^1 \times \b{E}_h^1} = -\half
\inprodf{f}{\projt_f \b{H}_h^2}{\b{n}^1 \times \b{E}_h^1},
\end{align*}
and thus $T_2+T_7=0$. Analogously  we obtain $T_3+T_6=0$.
\end{proof}
\subsection{A remark regarding dissipative stabilization}
The problem \eqref{eq:weakform_stabilized} can also be augmented with an
dissipative stabilization term, to obtain an upwind-type formulation (see
\cite{hesthaven_nodal_2002}, \cite{montseny_dissipative_2008}).
This is in particular advantageous for problems with strong singularities, or
when parts of the solution can be kept under-resolved. 
In this case, the additional stabilization term is given by
\begin{align*}
&D_h(\b{U}_h,\b{V}):=\inprodf{\Fi}{\alpha_E \jumpt{
\b{E}_h}}{\jumpt{\projt\b{v}}} + \inprodf{\Fb}{\alpha_E \jumpt{
\b{E}_h-\b{g}}}{\jumpt{\b{v}}}\\
&\inprodf{\Fi}{\alpha_H \jumpt{ \b{H}_h}}{\jumpt{\projt\b{w}}},\\
&\alpha_E=c\avg{\sqrt{\mu/\eps}}^{-1},\;\;\;
\alpha_H=c\avg{\sqrt{\eps/\mu}}^{-1}, \;\; c \geq 0.
\end{align*} 
However, this comes at the cost of adding
dissipation to the problem. Again assuming element-wise constant material 
parameters $\eps, \mu$, we obtain
\begin{align*}
&\mathcal{E}(t_{n+1}) - \mathcal{E}(t_{n})=\inprod{\b{J}}{\pvwE} -
\inprodf{\Fb}{\b{n}\times\b{g}}{\pvwH}\\
&-\inprodf{\Fi}{\alpha_E \jumpt{ \projt \b{E}_h}}{\jumpt{\projt\b{E}_h}} -
\inprodf{\Fb}{\alpha_E \jumpt{
\b{E}_h-\b{g}}}{\jumpt{\b{E}_h}}\\
&-\inprodf{\Fi}{\alpha_H \jumpt{\projt \b{H}_h}}{\jumpt{\projt\b{H}_h}},
\end{align*}
where the last three terms on the right-hand side are characteristic for the
dissipative formulation.
\section{Implementation}
For each time slab, \eqref{eq:weakform} yields a linear system of equations. Since
especially for three-dimensional problems, the direct solution becomes unfeasible due to
the large number of unknowns and high memory demands of sparse direct solvers, we resort
to an iterative solution of \eqref{eq:weakform}. Rather than assembling a matrix we
implement the evaluation of the residual
\begin{align}
\label{eq:residual}
R_h(\b{V})=B_h(\b{U}_h,\b{V})-L(\b{V})
\end{align}
directly. This is in particular advantageous in the context of adaptivity, where the discretization may change from time slab to time slab.\\
Furthermore under the assumptions of Theorem \ref{stability:boundSTL2:theorem}, we can derive a guaranteed error bound on the iteration error. This allows to balance discretization and iteration error, leading to greatly reduced computational costs.  The resulting method will, from the computational point of view behave similarly as explicit methods.
\subsection{Basis functions for trial- and test-space}
In order to allow for the efficient evaluation of the residual,  we chose the tensor product basis functions
\begin{align}
& \hat{\b{v}}_{cijkl}= l_i(\frac{t-t_k}{|I_k^K|}) \basiss{cjkl}\\
& \basiss{cjkl}=L_j(\hat{x})L_k(\hat{y})L_l(\hat{z})\,\b{e}_c\nonumber\\
&  i=0,...,p_t,\;\;j=0,...,p_x,\;\;k=0,...,p_y,\;\;l=0,...,p_z
\end{align}
for the local space $V_{h,\hat{K}}^k$.
Here $L_i(\xi)\,i=0,...,p$ denote the orthonormal Legendre polynomials on
$[0,1]$ and $l_i(t)$, the integrated Legendre polynomials on $[0,1]$:

\begin{align}
\label{def:ilegendre}
l_0(\xi)=1-\xi, \hspace{1cm} l_1(\xi)=\xi, \hspace{1cm} l_i(\xi)=\int\limits_{0}^\xi L_{i-1}(s) ds\;\;i=2,...,p\nonumber
\end{align}
Thus, the approximate solution in space-time element $I_k^K \times K$ can be expanded as $\restr{\b{E}_h}{I_k^K \times K}=DF^{-T} \hat{\b{v}}_{cijkl} e_{cijkl}$, where $e_{cijkl}$ denotes the coefficients.

For the local test-spaces $W_{h,\hat{K}}^k$ the basis is chosen to consist entirely of Legendre polynomials
\begin{align}
& \hat{\b{w}}_{cijkl}=L_i(\frac{t-t_k}{|I_k^K|}) \basiss{cjkl}\\
& \basiss{cjkl}=L_j(\hat{x})L_k(\hat{y})L_l(\hat{z})\,\b{e}_c\nonumber\\
&  i=0,...,p_t-1, j=0,...,p_x,k=0,...,p_y,l=0,...,p_z.
\end{align}

\subsection{Efficient evaluation of the space-time residual}
In the following, we outline how the above choices of basis for the trial
and test spaces lead to an efficient evaluation of the
space-time residual \eqref{eq:residual}.
In particular, for each space-time element,
the residual can be evaluated with optimal complexity of $\mathcal{O}(p^4)$
operations in the case of affine elements and elementwise constant $\eps,\mu$
and $\mathcal{O}(p^5)$ operations for non-affine elements, where for
simplicity, in the complexity analysis, we assume isotropic polynomial
degrees $p_t = p_x= p_y = p_z$.
In the following section we use index notation and summation convention.
\subsubsection*{Mass residual containing time derivatives}
Because of the tensor product structure of the local basis functions we can factor 
the space-time integral as follows
\begin{align*}
&R_E^{mass}\hspace{-3pt}(\b{v})=\inprodlocst{\varepsilon \p_t\b{E}_h}{\b{v}} =
\inprodlocst{\varepsilon \p_t  DF^{-T} \hat{\b{v}}_{cijkl} }{DF^{-T} \hat{\b{v}}_{dmnpq}} e_{dmnpq} \\
&= \int\limits_{I_k^K} L_i(\xi) \p_t l_m(\xi)d\xi \,\mathcal{M}_{cjkl,dnpq} \,e_{dmnpq},\\
&\mathcal{M}_{cjkl,dnpq}^\varepsilon=\int_{\hat{K}}\varepsilon DF^{-T} \basiss{cjkl}  DF^{-T} \basiss{dnpq} |J| d\hat{x} d\hat{y} d\hat{z}
\end{align*}
Note, that
we have $\int_0^1 L_i(\xi) \p_t l_m(\xi)d\xi=\delta_{im}$ for $i>1, m>1$ 
due to \eqref{def:ilegendre} and the orthogonality of the Legendre polynomials. 
Thus, evaluating the time derivative term corresponds to applying a spatial mass
matrix $\mathcal{M}_{cjkl,dnpq}$  $\mathcal{O}(p_t)$ times. 
Each application of $\mathcal{M}_{cjkl,dnpq}$ can be done with
$\mathcal{O}(p^3)$ operations for affine and $\mathcal{O}(p^4)$ operations 
for non-affine elements using fast summation techniques
\cite{melenk_fully_2001}, leading to a complexity of $\mathcal{O}(p^4)$ and
$\mathcal{O}(p^5)$ respectively.

\subsubsection*{Curl residual}
\begin{align*}
&R_H^{curl}\hspace{-3pt}(\b{v})=\inprodlocst{\curl \b{E}_h}{
\b{v}}\;\;\;R_{dmnpq}^{curl} =  \int\limits_{I_k^K} L_i(\xi) l_m(\xi)d\xi \,\mathcal{C}_{{cjkl,dnpq}}^{vol} \,e_{dmnpq}\\
&\mathcal{C}_{{cjkl,dnpq}}^{vol} = \int_{\hat{K}} \basiss{cjkl} \cdot \curl \basiss{dnpq} d\hat{x} d\hat{y} d\hat{z}
\end{align*}
Here, we also have linear complexity in $p_t$ for the number of volume-curl evaluations, since for $m>2$ there holds $l_m(\xi)=(L_{m+1}(\xi)-L_{m-1}(\xi))/\sqrt(2m+1)$. Note that the curl can also be evaluated with $\mathcal{O}(p^3)$ operations using recurrence relations for the derivatives of the Legendre polynomials (see e.g. \cite{koutschan_computer_2011}), such that a total complexity of $\mathcal{O}(p^4)$ is obtained.
\subsubsection*{Flux terms}
The flux terms are evaluated using fast summation techniques, such that in total 
$\mathcal{O}(p^4)$ operations are needed. For two space-time elements $I^1\times K^1,\,I^2\times K^2$ sharing the face $f_t \times f_s, f_t=I^1 \cap I^2, f_s=\p K^1 \cap \p K^2$, we consider the evaluation of the flux involving neighbor coupling.
Note, that in particular we allow for nonconforming interfaces in space and time. Recalling that we have $K^i=F_i([0,1]^3), I^i=\tau_i([0,1])$ and we use a co-variant transform for the spatial variables, we can write
\begin{align}
\label{implementation:fluxterm}
&R_H^{flux}\hspace{-3pt}(\b{v})=\int_{f_t}\int_{f_s} \b{E}_h^2\cdot(\b{n}^1
\times \projt_f \b{v}^1) \d{S}\d{t} = \nonumber\\
&\int_{{\tau_1}^{-1}(f_t)} \int_{{F_1}^{-1}(f_s)} {DF_1}^T {DF_2}^{-T} \circ \psi_1^2 \hat{\b{E}}_h^2 \circ \psi_1^2 \cdot (\hat{\b{n}}^1 \times \projt_f \hat{\b{v}}^1) \d{\hat{x}}\d{\hat{y}}\d{\hat{t}}.
\end{align}
Here $\psi_1^2: {{\tau_1}^{-1}(I^1) \times F_1}^{-1}(f_s) \rightarrow {{\tau_2}^{-1}(I^2) \times F_2}^{-1}(f_s)$ is the mapping from the reference coordinates of $I^1\times K^1$ to those of $I^2\times K^2$. The mappings  $\psi_1^2$ and $DF_1^TDF_2^{-T}$ are constant \cite{cohen_spatial_2006} and in general affine linear with a scaled permutation matrix. Note, that the scaling part is different from the identity matrix for nonconforming interfaces only. Permutations in the spatial reference coordinates occur for general hexahedral meshes.
However, for simplicity, we assume in the following that $\psi_1^2$ is of the form $(\hat{x}^2,\hat{y}^2,\hat{z}^2,\hat{t}^2)=(s_x \hat{x}^1 +b_x,s_y \hat{y}^1 +b_y, s_z \hat{z}^1 +b_t,s_t \hat{t}^1 +b_t)$, i.e. no coordinate permutations are present.
We have for a face with (reference) normal $\hat{\b{n}}^1=\b{e}_{3}$
\begin{align}
\label{implementation:fluxterm:facecoefficients}
\hat{\b{E}}_h^2 = l_i(\hat{t})L_j(\hat{x})L_k(\hat{y})L_l(-1) e_{cijkl} \b{e}_c = l_i(\hat{t})L_j(\hat{x})L_k(\hat{y}) e_{cjkl}^F \b{e}_c.
\end{align}
Obviously, $e_{cijk}^F=L_l(-1) e_{cijkl}$ can be computed with $\mathcal{O}(p^4)$ operations.
With $I_{\hat{t}}=\tau_1^{-1}(f_t)$ and $I_{\hat{x}} \times I_{\hat{y}} = F_1^{-1}(f_s)$, using the tensor-product structure of the trial  and testspace basis functions $\hat{\b{v}}_{cijkl}$ and $\hat{\b{w}}_{dmnop}$ respectively, \eqref{implementation:fluxterm} can be written as
\begin{align*}
&R_{dmnop}^{flux}=L_p(1) [\projt_f]_{mq} \int_{I_{\hat{t}}} l_i(s_t \hat{t}^1 +b_t)L_q(\hat{t}^1)\d{\hat{t}^1}  \int_{I_{\hat{x}}} L_j(s_x \hat{x}^1 +b_x) L_n(\hat{x}^1)\d{\hat{x}^1}\nonumber\\
&\times\,  \int_{I_{\hat{y}}} L_k(s_y \hat{y}^1 +b_y) L_o(\hat{y}^1) \d{\hat{y}^1} (\hat{\b{n}}^1 \times \b{e}_d) \cdot \b{e}_c e_{cijk}^F.
\end{align*}
Thus, the flux-residual $R_{dmnop}^{flux}$ can be evaluated as
\begin{align}
\label{implementation:fluxterm:summation}
&Aux1_{doij} =\int_{I_{\hat{y}}} L_k(s_y \hat{y}^1 +b_y) L_o(\hat{y}^1)\d{\hat{y}^1} (\hat{\b{n}}^1 \times \b{e}_d) \cdot \b{e}_c e_{cijk}^F\nonumber\\
&Aux2_{dnoi} = \int_{I_{\hat{x}}} L_j(s_x \hat{x}^1 +b_x) L_n(\hat{x}^1)\d{\hat{x}^1} Aux1_{doij}\nonumber\\
&Aux3_{dmno} = [\projt_f]_{mq} \int_{I_{\hat{t}}} l_i(s_t \hat{t}^1 +b_t)L_q(\hat{t}^1)\d{\hat{t}^1} Aux2_{dnoi}\nonumber\\
&R_{dmnop}^{flux} = L_p(1) Aux3_{dmno}.
\end{align}
The complexity of each summation in \eqref{implementation:fluxterm:summation} is $\mathcal{O}(p^4)$, in the case of a nonconforming interface in the respective direction. If in contrast the interface is conforming, the summation can be skipped due to the orthogonality properties of the trial and test basis functions.\\
In total the flux-residuals can be evaluated with $\mathcal{O}(p^4)$ operations. Note that the same applies to non-coupling terms also.

\subsection{Inexact iterative solution - guaranteed iteration error bound}
Solving the problem exactly can be very expensive. Instead we solve the problem inexactly and control the error introduced by the inexact solution.\\
Noting that the iteration error at solver iteration $m$ for timeslab $n$ fullfills
\begin{align*}
B_h(\b{U}_h^m-\b{U}_h,\b{V})=R_h^{m,n}(\b{V}).
\end{align*}
we can apply the stability estimate Theorem \ref{stability:boundSTL2:theorem} 
in order to obtain the guranteed bound on the errors $\b{e}^m=\b{E}_h^m-\b{E}_h$ 
and $\b{h}^m=\b{H}_h^m-\b{H}_h$ at time $t_N$
\begin{align}\label{eq::a_posteriori:iteration}
\|\eps^\half \b{e}^m(t_{N})\|_{\bL2(\Omega)}^2 + \|\mu^\half
\b{h}^m(t_{N})\|_{\bL2(\Omega)}^2 \leq  4\sum\limits_{n=1}^{N} \left(
(\frac{2t_N}{\epsmin}+ \frac{\Delta t^2}{2t_N
\epsmin})\|R_{h,E}^{m,n}\|_{W_h^\prime}^2 \right.\nonumber\\
\left.  + (\frac{2t_N}{\mumin} + \frac{\Delta
t^2}{2t_N\mumin})\|R_{h,H}^{m,n}\|_{W_h^\prime}^2 \right)=:\eta_{it}^2.
\end{align}
The dual norm of the residual can be evaluated exactly by computing its Riesz-representor, which  in this case, is just the application of a $L_2$-projection operator due to the entirely discontinuous test space.
\subsection{Remarks on the iterative solution proceedure}
Efficient preconditioning of the linear systems is currently an open problem.
Nevertheless, we would like to give some remarks
regarding this topic.\\
Since the linear system
governed by \eqref{eq:weakform_stabilized} is non-symmetric, we apply a
preconditioned GMRES solver with restarting after $n_{r}$ iterations, such
that, denoting with $N$ the number of degrees of freedom,  the memory requirement
for each timeslab is essentially $n_r \times N$. Usually we chose $n_r=10$. As
 preconditioner, we apply the time-derivative terms in
 \eqref{eq:weakform_stabilized} i.e.
\begin{align}\label{eq:preconditioner}
P(\b{U}_h,\b{V}):= \inprod{\varepsilon \p_t\b{E}_h}{\b{v}} + \inprod{\mu
\p_t\b{H}_h}{\b{w}}.
\end{align}
This choice of preconditioner has the advantage, that it is in the
worst-case block-diagonal with block size $3 p_x p_y p_z$.
Thus, the computational cost of one preconditioned GMRES-iteration associated with one
spatial element is in terms of computational cost comparable to $N_K$
time-steps of an explicit Runge-Kutta scheme with $p_t$ stages.
However, this simple preconditioner is effective only for sufficiently small
$\Delta t$.
This is due to the scaling with respect
to $\Delta t$ of the terms in \eqref{eq:weakform_stabilized} stemming from the
discretization of the curl-operator.
\subsubsection{Example}
In order to give an impression regarding the computational cost of the
iterative solution process, we report the number of iterations for a
$\mathrm{T}_{mm}$-mode in a cubic resonator in figure
\ref{fig:iterations_tm_mode} .
The resonator is discretized with $8\times8\times8$ hexahedral elements of degree $p=p_t=p_x=p_y=p_z$. We chose $\Delta
t=h/(2p+1)$, which corresponds to the maximal stable time-step reported in
\cite{sarmany_dispersion_2007} for an SSP low-storage Runge-Kutta method of
order $p+1$. Further we choose the wavenumber $m$ as 1, 2 and 3 for $p=1,2$,
$p=3,4$ and $p=5,6$ respectively. The exact solution of the linear systems
requires about $N_{it}=20-30$ iterations per time-slab. The resulting
computational costs are prohibitively high, in particular when compared to a
fully explicit scheme such as \cite{sarmany_dispersion_2007}, where the cost is
roughly equal to one iteration.
For inexact solution, the number of iterations is
reduced by a factor of 2-4.2 and is in the range of $\widetilde{N}_{it}=4-9$.
Furthermore, the iteration-error bound consistently overestimates the iteration
error with efficiency-index $I_{eff}=2-12$, such that the total error is
not increased by the inexact solution.
\begin{figure}
\begin{center}
\begin{tabular}{|l|l|l|l|l|l|l|l|}\hline
$p$ & steps & $N_{it}$ & $\widetilde{N}_{it}$ & speedup & $e(T)$
& $\widetilde{e}(T)$ & $I_{eff}$\\
\hline
$1$ & $189$ & $7.9$ & $4.0$ & $2.0$ & $2.04\times 10^{-1}$ & $2.04\times
10^{-1}$ & $1.66$ \\
$2$ &$315$ & $24.0$ & $6.1$ & $4.0$ & $3.65\times 10^{-3}$ & $3.90\times
10^{-3}$ & $3.12$ \\
$3$ &$220$ & $25.0$ & $6.0$ & $4.2$ & $5.57\times 10^{-3}$ & $5.31\times
10^{-3}$ & $6.75$ \\
$4$ &$283$ & $20.0$ & $7.8$ & $2.6$ & $2.30\times 10^{-4}$ & $2.26\times
10^{-4}$ & $12.58$ \\
$5$ &$213$ & $22.0$ & $8.6$ & $2.5$ & $1.91\times 10^{-4}$ & $1.83\times
10^{-4}$ & $5.67$ \\
$6$ &$252$ & $18.0$ & $8.8$ & $2.0$ & $4.84\times 10^{-5}$ & $4.83\times
10^{-5}$ & $5.70$\\
\hline
\end{tabular}

\caption{Number of time-steps, average number of GMRES-iterations $N_{it}$
per time-step for exact and $\widetilde{N}_{it}$ for inexact solution of the
linear systems, speedup in terms of total iterations, errors
$e(T)=\|\b{U}(T)-\b{U}_h(T)\|_{\bL{2}(\Omega)}$ and
$\widetilde{e}(T)=\|\b{U}(T)-\widetilde{\b{U}}_h(T)\|_{\bL{2}(\Omega)}$ for
exact and ineaxct solution and efficiency index for the iteration error bound
\eqref{eq::a_posteriori:iteration} $I_{eff}=\eta_{it}/\|\b{e}_h^n(T)-\b{e}_h(T)\|_{\bL{2}(\Omega)}$}
\label{fig:iterations_tm_mode}
\end{center}
\end{figure}
\section{Numerical Experiments}
The first two numerical experiments feature basic academic examples for assessing convergence properties when local refinement, especially with respect to time, is present in the discretization.
\subsection{Convergence tests on non-adaptive discretizations}
\subsubsection{TM Mode $hp$-discretization}
The initial data is chosen such that the $\mbox{TM}_{mn}$ mode is approximated in $\Omega=[0,1]\times[0,1]\times[0,1/5]$.
We choose $m=n=1$ leading to frequency $\omega=\pi\sqrt{m^2+n^2}$.
The exact solution for the electric field is:
\begin{equation}
\b{E}=\sin(m \pi x) \sin(n \pi y)\cos(\omega t) \b{e}_z
\end{equation}
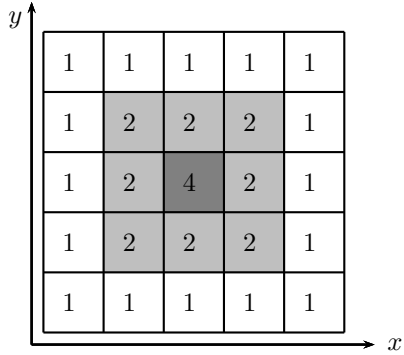
\begin{figure}[h!]
\centering
\psset{unit=0.8cm}
\begin{pspicture}(-1,0)(6,5.5)

\psaxes[labels=none,ticks=none]{->}(-0.2,-0.2)(-0.2,-0.2)(5.5,5.5)
\psframe[fillstyle=solid,fillcolor=lightgray,linewidth=0pt](1,1)(4,4)
\psframe[fillstyle=solid,fillcolor=gray,,linewidth=0pt](2,2)(3,3)

\uput[0](-0.8,5.2){$y$}
\uput[0](5.5,-0.2){$x$}
\psline(0,5)(5,5)
\psline(0,4)(5,4)
\psline(0,3)(5,3)
\psline(0,2)(5,2)
\psline(0,1)(5,1)
\psline(0,0)(5,0)

\psline(5,0)(5,5)
\psline(4,0)(4,5)
\psline(3,0)(3,5)
\psline(2,0)(2,5)
\psline(1,0)(1,5)
\psline(0,0)(0,5)

\uput[0](0.1,0.5){$1$}
\uput[0](1.1,0.5){$1$}
\uput[0](2.1,0.5){$1$}
\uput[0](3.1,0.5){$1$}
\uput[0](4.1,0.5){$1$}

\uput[0](0.1,1.5){$1$}
\uput[0](0.1,2.5){$1$}
\uput[0](0.1,3.5){$1$}
\uput[0](4.1,1.5){$1$}
\uput[0](4.1,2.5){$1$}
\uput[0](4.1,3.5){$1$}

\uput[0](0.1,4.5){$1$}
\uput[0](1.1,4.5){$1$}
\uput[0](2.1,4.5){$1$}
\uput[0](3.1,4.5){$1$}
\uput[0](4.1,4.5){$1$}

\uput[0](1.1,1.5){$2$}
\uput[0](2.1,1.5){$2$}
\uput[0](3.1,1.5){$2$}

\uput[0](1.1,3.5){$2$}
\uput[0](2.1,3.5){$2$}
\uput[0](3.1,3.5){$2$}

\uput[0](1.1,2.5){$2$}
\uput[0](3.1,2.5){$2$}

\uput[0](2.1,2.5){$4$}

\end{pspicture}
\caption{The spatial mesh. Numbers inside  elements denote the number of local time steps. The polynomial degrees are chosen isotropically (in space and time), and increase from $p=p_{min}$ to $p=p_{min}+2$ from the boundary to the center of the computational domain according to the shading.}\label{fig:TM11hp:discretization}
\end{figure}

The time traces of the spatial $L_2$-error are depicted for 200 periods of the solution in Fig. \ref{fig:TM11hp:error}. Furthermore, in Fig. \ref{fig:TM11hp:error} one can observe exponential convergence in the norm $\|\cdot\|_{L_2(\I_n;\bL{2}(\Omega))}$.

\begin{figure}[h!]
\centering
\begin{tabular}{cc}
\hspace{-10mm}\includegraphics[scale=0.33]{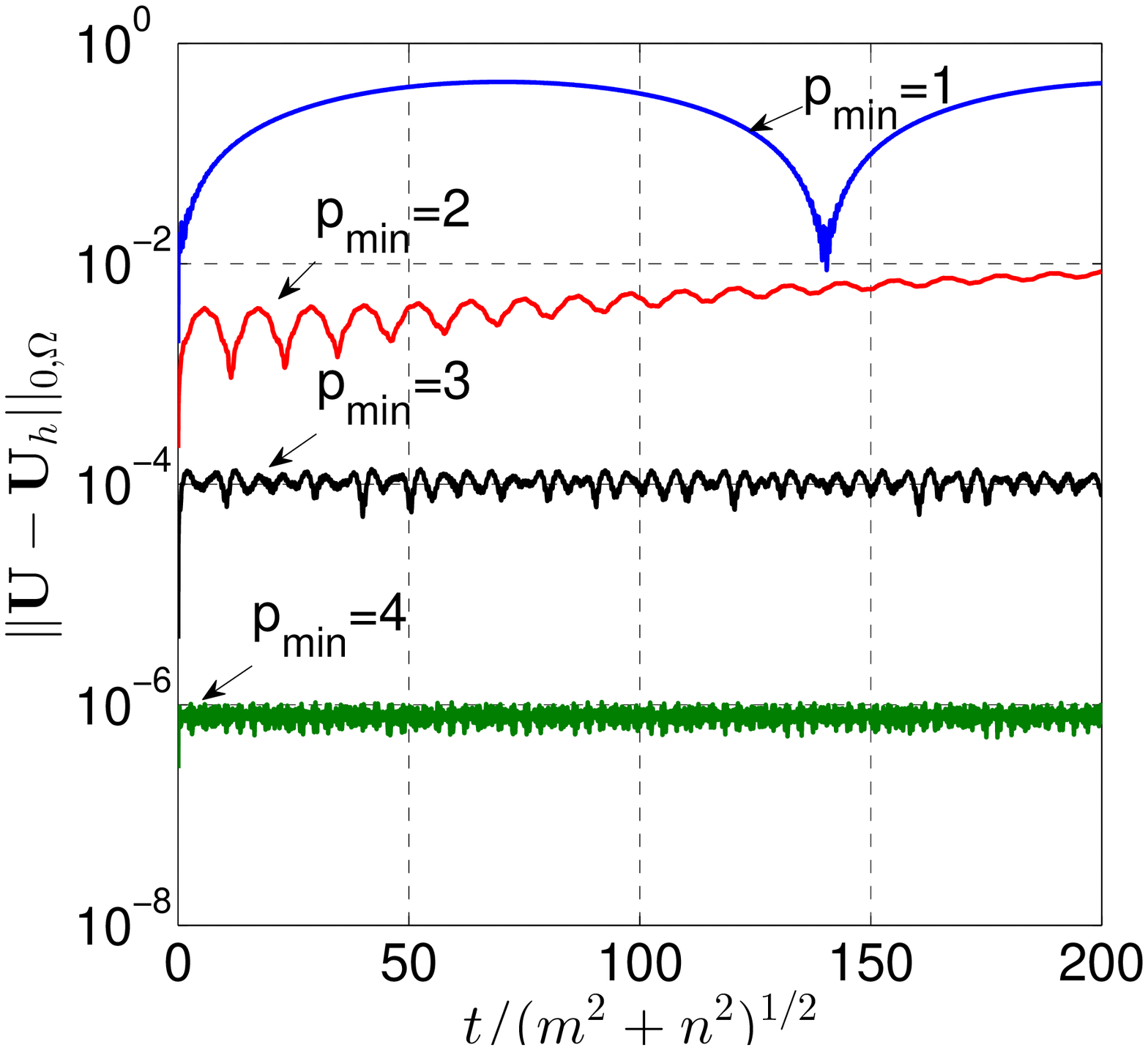}\hspace{-4mm} & \hspace{-4mm}\includegraphics[scale=0.33]{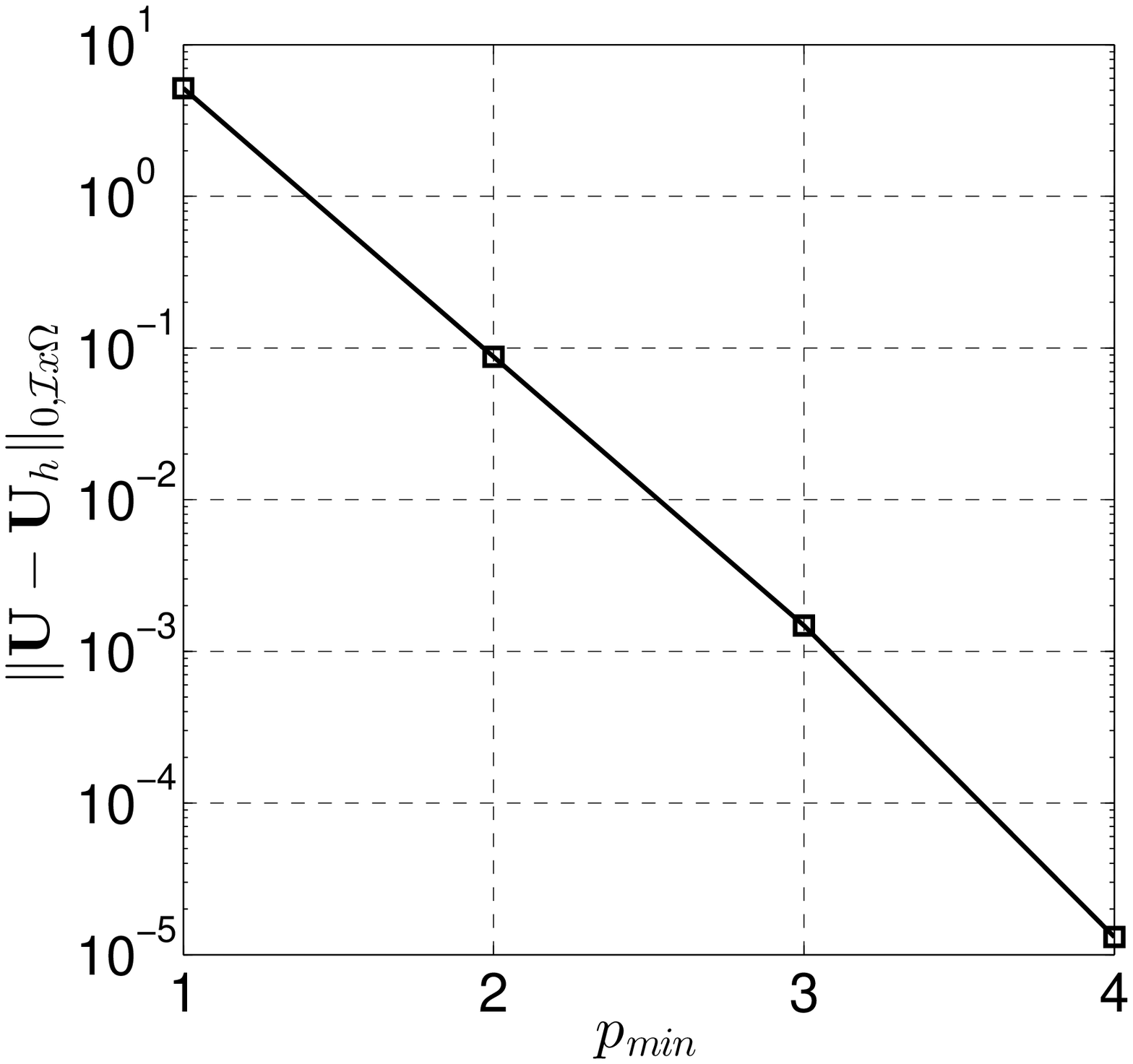}
\end{tabular}
\caption{Left: Temporal evolution of the spatial $L_2$-error, right: error $\|\b{U}-\b{U}_h\|_{L_2(\I;\bL{2}(\Omega))}$ for $p_{min}=1,2,3,4$ and the setup depicted in Fig. \ref{fig:TM11hp:discretization}}
\label{fig:TM11hp:error}
\end{figure}

\FloatBarrier
\subsubsection{Local $h$-refinement in time}
In order to exploit the temporal accuracy of the method in the case of local $h$-refinement 
in time direction, the source term and boundary conditions are chosen such that 
the exact solution \cite{verwer_composition_2012} is
\begin{align*}
\b{E}&=e^t x(x-1)z(1-z)\b{e}_y,\\
\b{H}&=e^t x(x-1)(1-2z)\b{e}_x - e^t (2x-1)z(1-z) \b{e}_z.
\end{align*}
The problem is solved on the space-time domain $\I\times\Omega$ with
$\I=[0,5], \Omega=[0,1]^3$, which is subdivided in time slabs $\Delta t\times\Omega$. 
The spatial mesh and the temporal refinement level associated with each spatial 
cell is shown in Fig. \ref{fig:temporal_accuracy:grid}. The solution is 
approximated with quadratic polynomials in the $x-z$ directions, such that the 
temporal error is expected to be dominating. In the temporal direction the
polynomial degree is set to $p_t$. In Fig. \ref{fig:temporal_accuracy:convergence} 
left, one can observe that the error in the norm $\|\cdot\|_{L_2(\I;\bL{2}(\Omega))}$ 
is of order $p_t+1$. If we instead consider the quantity  
$\max\limits_k \|\b{U}(t_k)-\b{U}_h(t_k)\|_{\bL{2}(\Omega)}$,  an order of
$2p_t$ can be observed in Fig. \ref{fig:temporal_accuracy:convergence} right. 
For continuous-Galerkin (cG) time stepping schemes, this nodal superconvergence
behaviour is reported in \cite{akrivis_galerkin_2011}.
\begin{figure}[h!]
\centering
\includegraphics[scale=0.3]{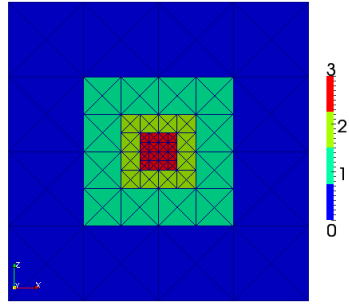}
\caption{Temporal refinement level}
\label{fig:temporal_accuracy:grid}
\end{figure}
\begin{figure}[h!]
\centering
\begin{tabular}{cc}
\hspace{-12mm}\includegraphics[scale=0.32]{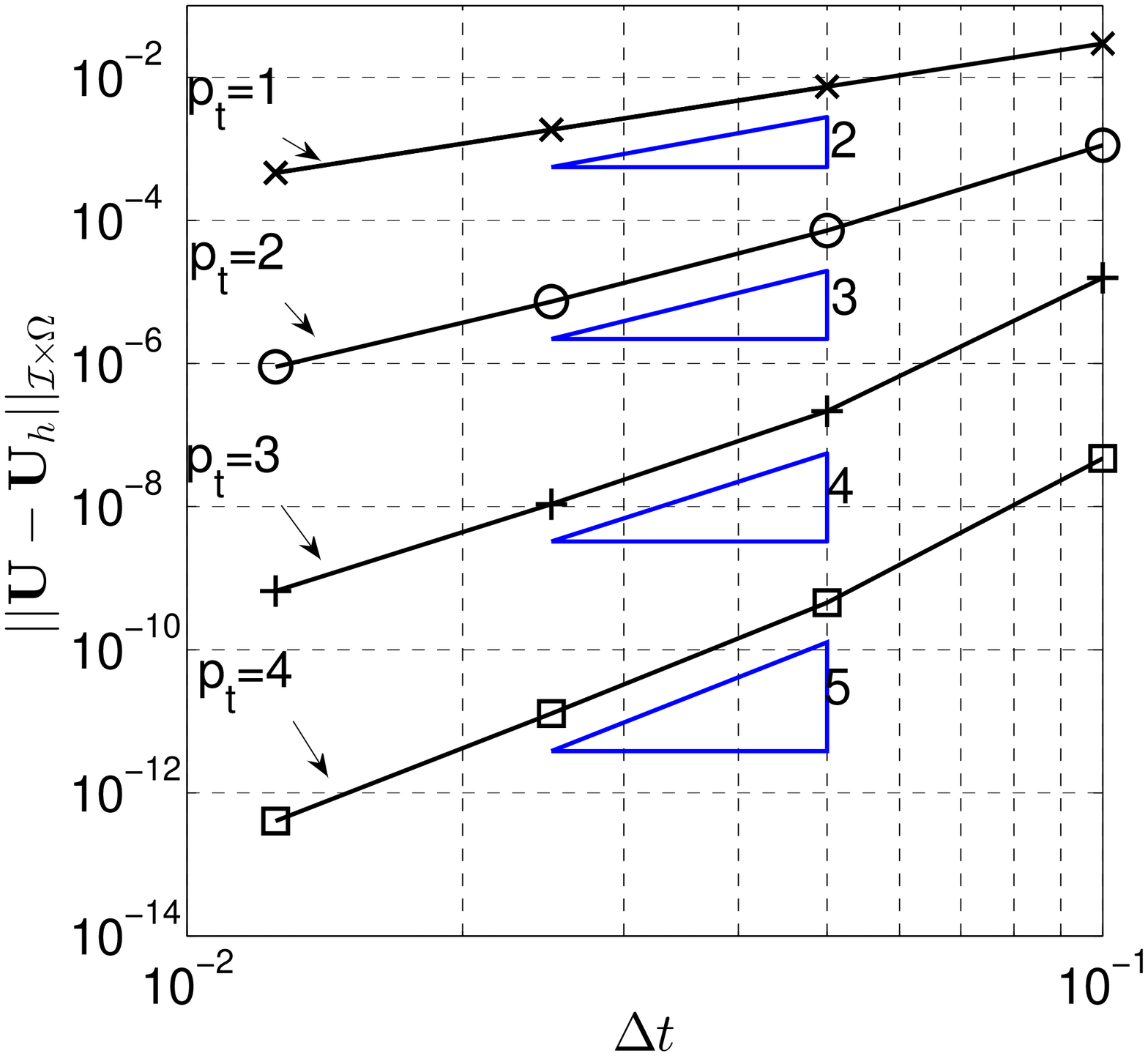}\hspace{-4mm} & \hspace{-4mm}\includegraphics[scale=0.32]{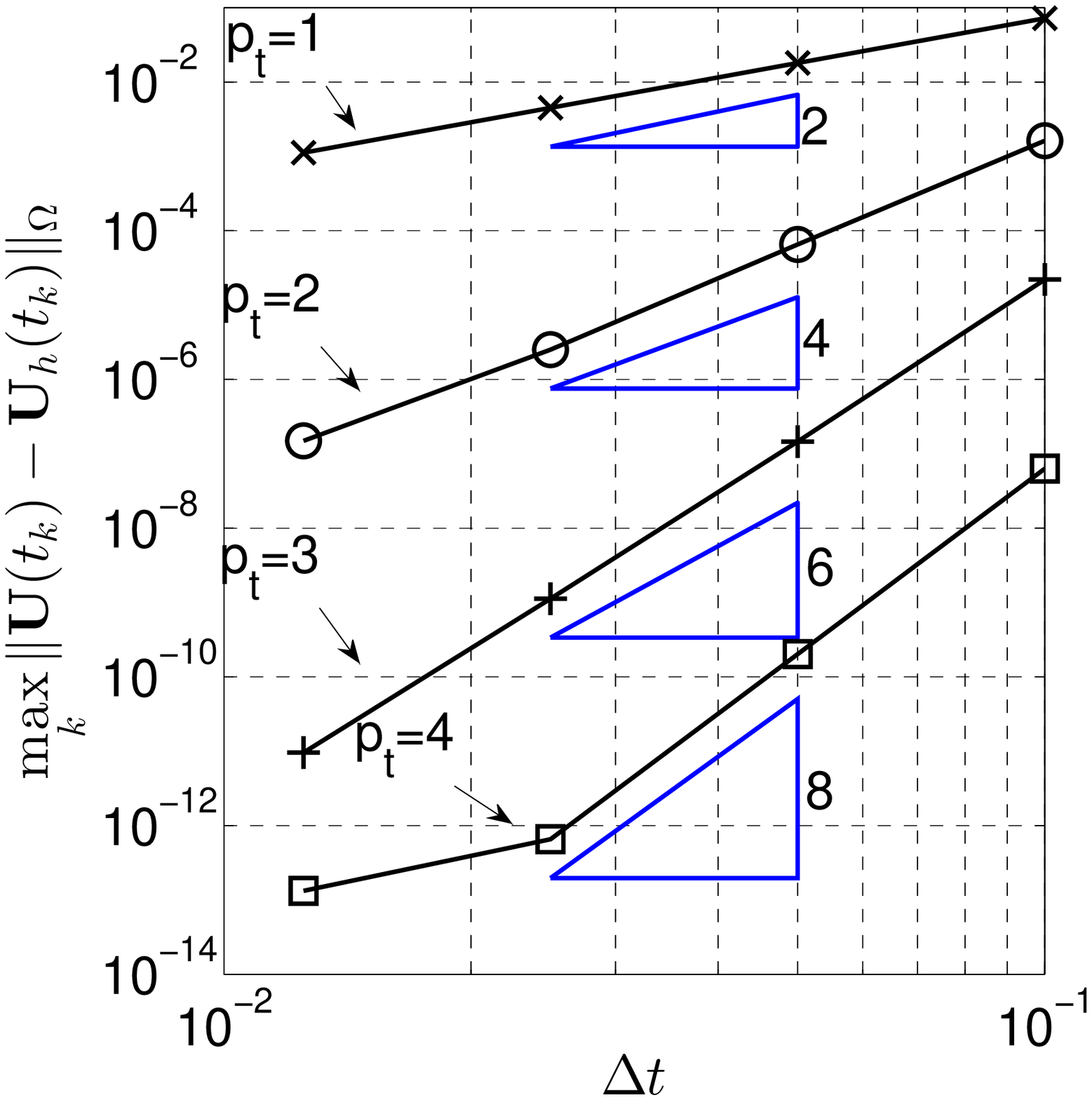}
\end{tabular}
\caption{Left: error $\|\b{U}-\b{U}_h\|_{L_2(\I;\bL{2}(\Omega))}$, right: error $\max\limits_k \|\b{U}(t_k)-\b{U}_h(t_k)\|_{0,\Omega}$ for the local $h$-refinement with respect to time depicted in Fig.\ref{fig:temporal_accuracy:grid}}
\label{fig:temporal_accuracy:convergence}
\end{figure}

\FloatBarrier
\subsubsection{Broadband pulse in a waveguide}
\label{sec:coaxial_waveguide_uniform_p}
We consider a broadband pulse in a coaxial waveguide. The exact solution is
\begin{align*}
\b{E}&=\frac{1}{r} e^{-(\pi(f_2-f_1)(z-t)/2)^2} \sin (\pi (f_1+f_2) (z-t)) \b{e}_r \\
\b{H}&=\frac{1}{r} e^{-(\pi(f_2-f_1)(z-t)/2)^2} \sin (\pi (f_1+f_2) (z-t)) \b{e}_{\varphi}.
\end{align*}
Here $r$, $\varphi$ and $z$ denote the radial, azimuthal and axial coordinates
respectively and $\b{e}_r$, $\b{e}_\varphi$ and $\b{e}_z$ the corresponding unit
vectors. The mid-frequency $(f_1+f_2)/2$ is chosen such that the corresponding
wavelength is approximately $1/16$ the length of the waveguide. The spatial mesh
consists of 1600 hexahedra with an edge-length of approximately $1/3 \lambda$.
The global time step is fixed, such that in total $400$ time steps are
neccessary to propagate the pulse through the entire waveguide. In Fig.
\ref{fig:coaxial_waveguide_uniform_p:error_vs_time} we show the time trace of
the spatial $L_2$-errors. The error exhibits an odd even pattern, the origin of
which is yet unknown. We suspect this behavior to be related to the choice of a
central flux. A similar behaviour is reported in \cite{hesthaven_nodal_2007}.
 \begin{figure}[h!]
\centering
\includegraphics[scale=0.33]{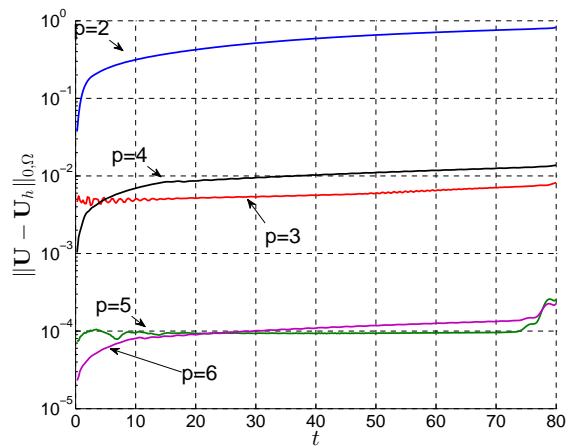}
\caption{Temporal evolution of the spatial $L_2$-error for $p=2,3,4,5,6$}
\label{fig:coaxial_waveguide_uniform_p:error_vs_time}
\end{figure}

\FloatBarrier
\subsubsection{Bi-static RCS, metallic sphere - global $p$-refinement}
A metallic sphere of radius $a=1$ is illuminated by a $\b{e}_x$-polarized plane wave traveling in $e_z$-direction with wave number $\omega=2\pi$. We chose $\Omega$ as a spherical shell with outer radius $R=5$. For the outer boundary a Silver-M\"uller boundary condition is applied. The spatial grid consists of only 576 hexahedral elements with about 2 elements per wavelength at the surface of the scatterer and one element per wavelength at the absorbing boundary. The polynomial degrees are chosen as $p=p_t=p_x=p_y=p_z=p_{geo}$, where $p_{geo}$ is the polynomial degree of the elemental mapping $F$.  We compute the bi-static RCS
\begin{align*}
\sigma(\phi,\theta)=\lim\limits_{r\rightarrow\infty}4\pi r^2\frac{|\b{E}_{sc}(r,\phi,\theta)|^2}{|\b{E}_{inc}(r,\phi,\theta)|^2}
\end{align*}
Here $r$ denotes the distance from the center of the sphere to the point of observation $r\b{x}$, $\phi$ and $\vartheta$ the azimuthal and polar angle between the wave and $\b{x}$.
The scattered far field is evaluated at a closed surface $S$, one wavelength away from the scatterer by evaluating the near-to-far field transformation
\begin{align*}
\b{E}_{sc}=\frac{e^{-i\omega r}}{4\pi r} \int\limits_{t_0}^{t_1}\int_S \left[ \b{x}\times(\b{x}\times(\b{n}\times\b{H}_h)) + \b{x}\times(\b{E}_h\times\b{n})\right]e^{i\omega\b{x}\cdot\b{y}-i\omega t} \d{S}(\b{y})\d{t}
\end{align*}
with Gauss quadrature of sufficiently high order. The computed RCS in Fig. \ref{fig:scattering_sphere_pec:rcs_p_refinement_global} converge quickly to the analytical Mie-series solution.
\begin{figure}[h!]
\centering
\includegraphics[scale=0.4]{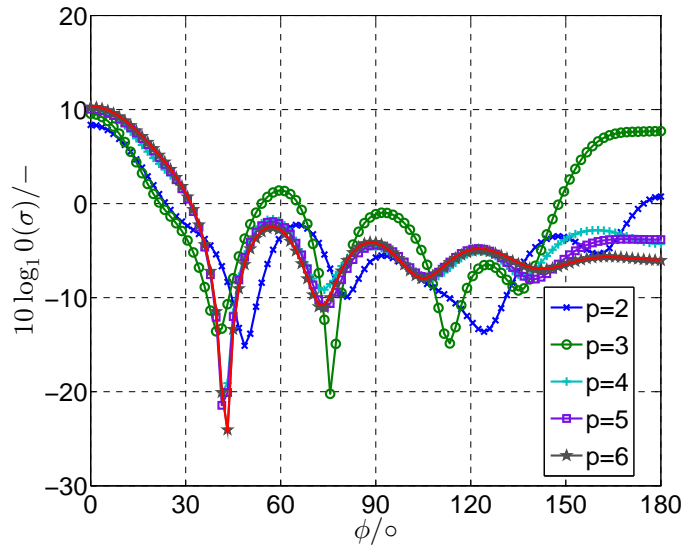}
\caption{Bi-static RCS,  of a $ka=2\pi$ PEC sphere, numerical solutions for a
576-element mesh with polynomial degrees $p=2,\ldots,6$ and the
analytical Mie-series solution (red)}
\label{fig:scattering_sphere_pec:rcs_p_refinement_global}
\end{figure}
\FloatBarrier
\subsection{Space-time $hp$-adaptive examples}

In order to demonstrate the method's capability for space-time $hp$-refinement, we propose an adaptive algorithm, which extends the concept of reference solutions \cite{demkowicz_computing_2007,solin_higher-order_2003} to the space-time context.
\begin{itemize}
  \item First a global time step $\Delta t$, resulting in equally sized time slabs, is chosen.
  \item Given a time slab $\Delta t\times\Omega$, an initial mesh $\mathcal{S}_0$ and polynomial-degree distribution $\b{p}_0$
  \item Define a coarse discretization $(\mathcal{S}_H,\b{p}_H)=(\mathcal{S}_0,\b{p}_0)$  and a refined discretization $(\mathcal{S}_h,\b{p}_h)$ with finite element spaces $V_H$ and $V_h$ respectively. The refined discretization is obtained by isotropically refining all space-time elements and increasing the polynomial degrees by one. Note that $V_H$ is contained in $V_h$. \\
  Then the following steps are performed iteratively
  \begin{enumerate}
    \item SOLVE: Solve the problem \eqref{eq:weakform} on the coarse and fine discretizations.
    \item ESTIMATE: Compute the error indicators $\eta(\I_n\times
    K)=\|\b{U}_h-\b{U}_H\|_{L_2(\I_n;\bL{2}(\Omega))}$ and the
    approximate error bound $\|\b{U}_h-\b{U}_H\|_{L_2(\I_n;\bL{2}(\Omega))} \leq TOL$ stop
    and proceed to the next time slab. Otherwise,
    \item MARK: Apply a fixed fraction marking strategy \cite{dorfler_convergent_1996} based on the indicators $\eta(\I_n\times K)$
    \item REFINE: For each marked $\I_n\times K$, set up a list of refinement candidates $V_{H,\I_n\times K}^{\mathrm{cand}}(\I_n\times\hat{K},\b{p}_K)$. In particular, we allow for all combinations of the following modifications of the discretization parameters:
    \begin{itemize}
      \item raise/decrease the polynomial degrees $p_x,p_y,p_z,p_t$
      \item isotropically $h$-refine/derefine in the spatial directions
      \item increase/decrease the temporal refinement level
    \end{itemize}
    In case of an refinement, which yields new space-time elements, we restrict the number of candidates by choosing identical polynomial degrees for each new element.\\
    Choose
    \begin{align*}
V_{H,\I_n\times K} = \argmin \limits_{V_H^{\mathrm{cand}}
S.T.\eta-\eta^{\mathrm{cand}}>0} \frac{\eta^{\mathrm{cand}} -\eta
}{\#DOF(V_H^{\mathrm{cand}})},
\eta^{\mathrm{cand}}\\
=\|\b{U}_h-\Pi_{V_h}^{V_H^{\mathrm{cand}}}\b{U}_h\|_{L_2(\I_n;\bL{2}(\Omega))}^2
\end{align*}
as the new local finite elements space, leading to a new global space $V_H$.  Build a new $V_h$ from $V_H$ and go back to SOLVE.
  \end{enumerate}
  \item Now, for the current time slab the final coarse and fine grid solutions have been obtained.
\end{itemize}
\emph{Remark1:}
The initial data for coarse- and fine-grid solves for the current time slab is taken as the $L_2$-projection of the refined solution $U_h$ from the previous time slab.
Note that in the case of spatial derefinement with respect to the previous time slab, i.e. if the spatial part of the finite element space is not contained in the current one, there will be dissipation introduced by the projection. See \cite{schnepp_efficient_2011}, for a more detailed discussion.
However the amount of dissipation introduced seemed to be negligible compared to the total error for the examples we have considered. \\
\emph{Remark2:}
If an iterative solver is applied, we can choose the coarse grid solution as starting point for solving the fine grid problem.
 We have observed, that this considerably cuts down the number of fine grid iterations.

\subsubsection{Broadband pulse in coaxial waveguide}
\begin{figure}[h!]
\centering
\psset{unit=0.8cm}
\begin{pspicture}(-1,-1)(5,5)
\psaxes[labels=none,ticks=none]{->}(-0.2,-0.2)(-0.2,-0.2)(1,1)
\pspolygon[fillcolor=lightgray,fillstyle=solid](0,0)(4,0)(2,2)(0,0)
\pspolygon[fillcolor=lightgray,fillstyle=solid](0,4)(4,4)(2,2)(0,4)
\psline(0,0)(4,0)
\psline(4,0)(4,4)
\psline(4,4)(0,4)
\psline(0,4)(0,0)
\psline(0,0)(4,4)
\psline(0,4)(4,0)
\uput[0](1.5,1){$p_y$}
\uput[0](0.5,2){$p_x$}
\uput[0](-1,0.8){$y$}
\uput[0](0.8,-0.5){$x$}
\end{pspicture}
\caption{Visualization of tensor product polynomials with degrees $p_x$,$p_y$}\label{fig:tensor_product_polynomial_visualization}
\end{figure}
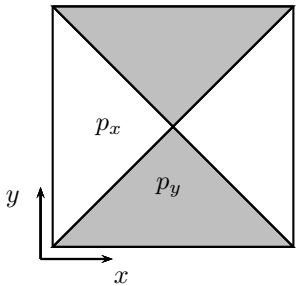
We have solved the example from \ref{sec:coaxial_waveguide_uniform_p} with the 
$hp$-adaptive algorithm. For the initial discretization we have chosen a coarse
mesh consisting of 200 elements of degree $p_t=1, p_x=p_y=p_z=0$. In Fig.
\ref{fig:coaxial_waveguide_adaptive_hp:discretization_step86} we visualize the coarse grid $hp$-discretization for time slab 86, by showing from top to
bottom the electric field magnitude, the spatial polynomial discretization,
using the tensor product visualization depicted in Fig.
\ref{fig:tensor_product_polynomial_visualization} and the temporal polynomial
degree distribution. In axial direction, where the pulse has greater variation, spatial polynomial degrees are chosen to be larger than in radial direction. The temporal polynomial degrees are raised in the area where the pulse is situated. The temporal refinement level was not raised in any of the time slabs.
\begin{figure}[h!]
\centering
\includegraphics[scale=0.3]{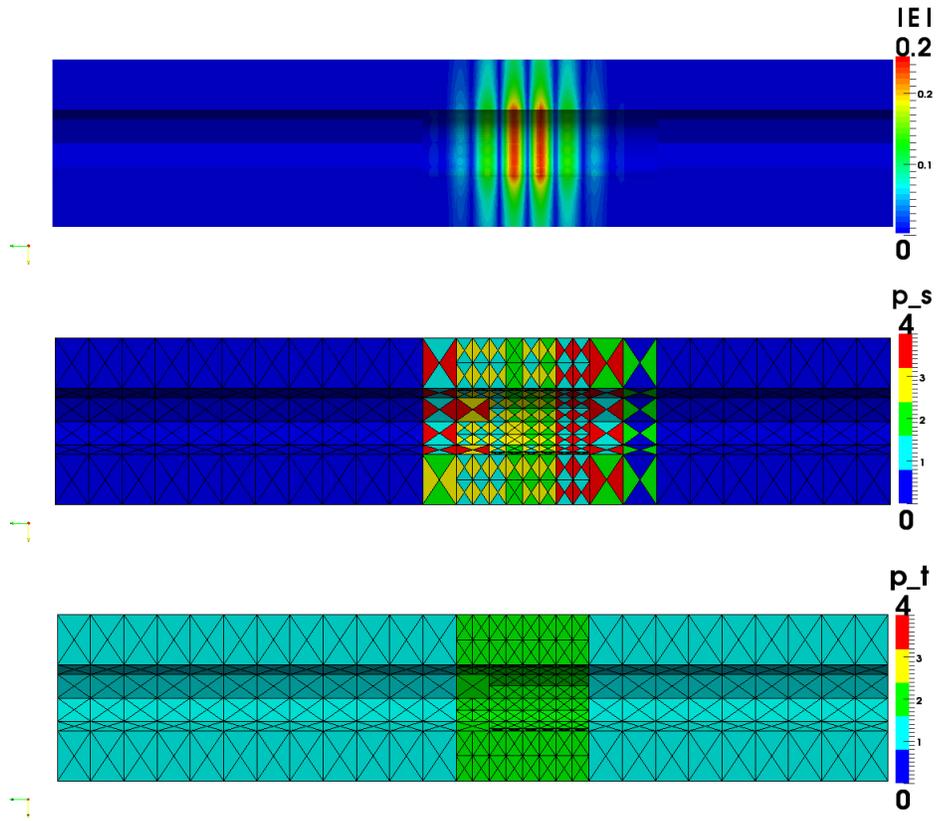}
\caption{Propagation of a broadband pulse in a cylindrical waveguide. From top to bottom: distribution of $|\b{E}_h|$, $(p_x,p_y,p_z)$, and $p_t$}
\label{fig:coaxial_waveguide_adaptive_hp:discretization_step86}
\end{figure}
For the $hp$-adaptive solution we have obtained a relative fine-grid error in the norm $\|\cdot\|_{L_2(;\bL{2}(\Omega))}$ of $9.6e-3$.

\subsubsection{Scattering of a dielectric sphere}
In the following, we consider the scattering of a Gaussian plane-wave by a
dielectric sphere with $\eps=4$. The problem again was solved using the $hp$-adaptive algorithm. 
For the initial discretization we have chosen a coarse
mesh consisting of 72 elements of degree $p_t=1, p_x=p_y=p_z=0$.
The discretization of time slabs 33, 65 and 124 (from top to bottom) 
is depicted in Fig. \ref{fig:scattering_sphere_dielectric_adaptive_hp:discretization_step34}. The fact that the electric field has limited regularity at the material interface is reflected by the choice of moderate polynomial degrees in space and linear polynomial degrees in time and a comparatively small spatial mesh size near the material interface. Refinement in the temporal direction was almost exclusively $h$-refinement, i.e. the temporal polynomial degree was almost exclusively equal to one. Thus we do not show the temporal polynomial degree distributions.\\
The bi-static RCS for $ka=1$ is depicted in Fig. \ref{fig:scattering_sphere_dielectric_adaptive_hp:rcs_ka1}, where a relative $\ell_2$-error  of $2.53 \cdot 10^{-2}$ was obtained.
\begin{figure}[H]
\centering
\includegraphics[scale=0.3]{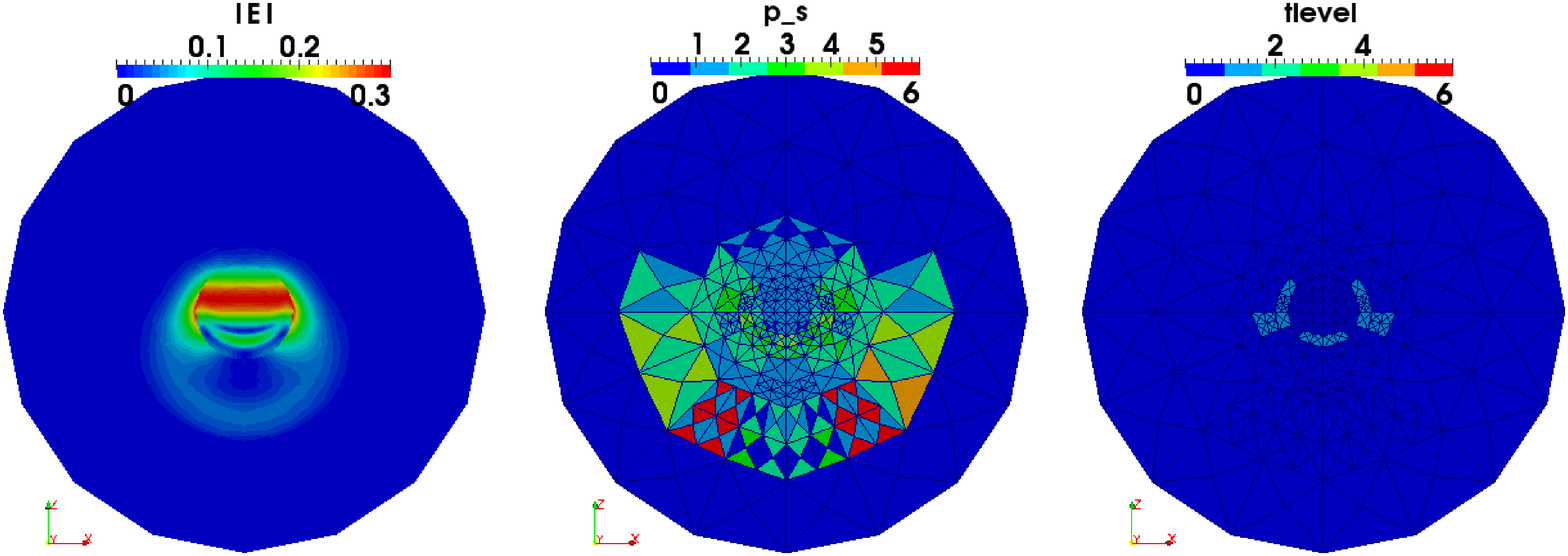}\\
\includegraphics[scale=0.3]{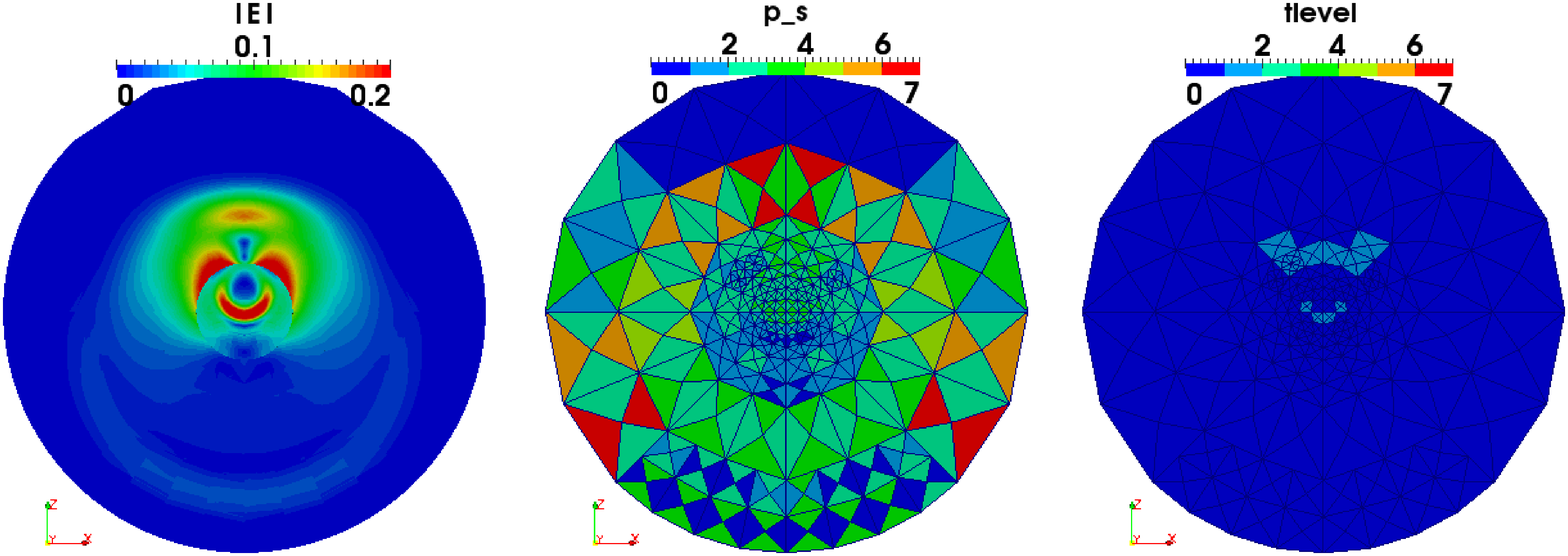}\\
\includegraphics[scale=0.3]{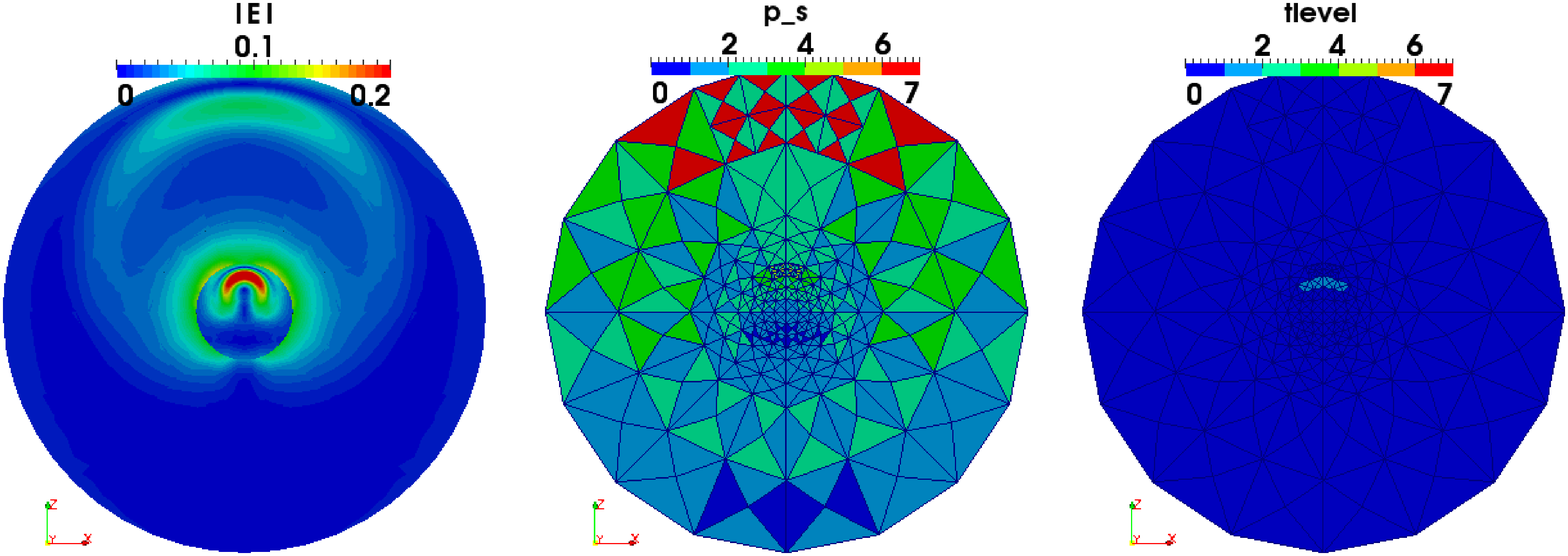}
\caption{$hp$-adaptive simulation of the scattering from a dielectric sphere. From left to right: distribution of $|\b{E}_h|$, spatial polynomial degrees $(p_x,p_y,p_z)$ and temporal refinement level at timeslab 34, 66, 125 (from  top to bottom)}
\label{fig:scattering_sphere_dielectric_adaptive_hp:discretization_step34}
\end{figure}
\begin{figure}[h!]
\centering
\includegraphics[scale=0.33]{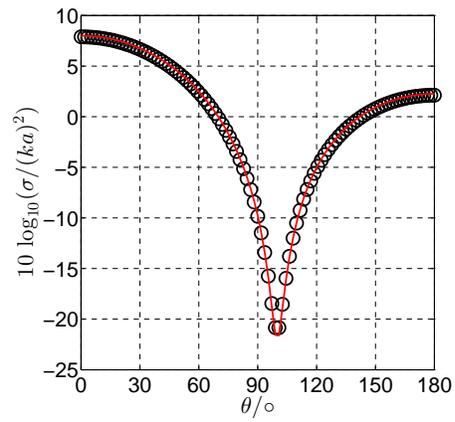}
\caption{Bistatic RCS of a dielectric sphere with $\eps=4$ and $ka$=1,
numerical solution (black circles), analytical Mie-series solution (red line)}
\label{fig:scattering_sphere_dielectric_adaptive_hp:rcs_ka1}
\end{figure}
\FloatBarrier
\section{Conclusions}
We have devised a space-time Galerkin method, which allows for local $hp$-refinement 
in space and time by treating the spatial part of the discretization with a DG approach, 
wheras the temporal part is treated with a continuous Galerkin approach. 
The resulting implicit method can be shown to be non-dissipative, as long as the 
spatial part of the discretization is kept constant from time-slab to time-slab. 
We have shown, that the method can be implemented, such that the complexity of the 
residual evaluation for an iterative solution is $\mathcal{O}(p^4)$ for affine elements 
and $\mathcal{O}(p^5)$ for non affine elements. Furthermore, for the case of no 
local refinement with respect to time, we have devised an \emph{a posteriori} 
bound on the iteration error. Thus a balancing of the iteration and discretization 
errors is possible, provided that an \emph{a posteriori} estimate for the 
discretization error is available. While the presented method has higher
computational costs than explicit hp-DG methods such as
\cite{schnepp_efficient_2011,schnepp_error_2014}, it provides
the possibility to apply hp-adaptivity not only in space but also in time. 
It therefore allows to control the approximation error in the entire space-time
domain of interest. We presented numerical experiments confirming that the
method can be used for fully space-time hp-adaptive simulations.\\
The \emph{a priori}
and \emph{a posteriori} error analysis is subject of ongoing work.
Other future work should relate to the investigation of more efficient methods
for the iterative solution of the linear systems, in particular regarding
preconditioning.
 	 \section*{References}
\bibliographystyle{elsarticle-num}
\bibliography{bibdatabase}
\section*{Acknowledgements}
The work of M. Lilienthal is supported by the 'Excellence Initiative' of the German Federal and State Governments and the Graduate School of Computational Engineering at Technische Universit\"at Darmstadt and the DFG under grant no. DFG WE 1239/27-2.\\
S.M. Schnepp acknowledges the support of the Alexander von Humboldt Foundation through a Feodor Lynen-Research Fellowship.\\
We thank Jens Niegemann for providing code for the computation of the Mie-series
solutions.\\
We would like to thank Herbert Egger for giving valuable hints, which improved the
presentation of the material significantly.
\end{document}